\documentclass[11pt]{amsart}
\usepackage{amsthm}
\usepackage{amssymb}
\usepackage{mathtools}
\usepackage{pdfpages}

\newcommand{\Mp}{\mathfrak{M}}
\newcommand{\BrAD}{\mathfrak{B}_{\mathrm{add}}}
\newcommand{\BrMU}{\mathfrak{B}_{\mathrm{mul}}}
\newcommand{\Bc}{\mathfrak{C}}
\newcommand{\Cs}{\mathfrak{S}}
\newcommand{\N}{\mathbb{N}}
\newcommand{\Z}{\mathbb{Z}}

\newcommand{\C}{\mathbb{C}}
\newcommand{\D}{\mathbb{D}}
\newcommand{\Sym}{\mathbb{S}}
\newcommand{\Alt}{\mathbb{A}}
\newcommand{\Aut}{\operatorname{Aut}}
\newcommand{\Hom}{\operatorname{Hom}}
\newcommand{\Hol}{\operatorname{Hol}}

\newcommand{\PSL}{\mathrm{PSL}}

\newcommand{\GL}{\mathrm{GL}}
\newcommand{\id}{\mathrm{id}}
\newcommand{\Soc}{\mathrm{Soc}}
\newcommand{\Add}[1]{(#1,\cdot)}
\newcommand{\Mul}[1]{(#1,\circ)}

\makeatletter
\numberwithin{equation}{section}
\numberwithin{figure}{section}
\numberwithin{table}{section}
\theoremstyle{plain}
\newtheorem{thm}{Theorem}[section]
\theoremstyle{plain}
\newtheorem{lem}[thm]{Lemma}
\newtheorem{cor}[thm]{Corollary}
\theoremstyle{plain}
\newtheorem{pro}[thm]{Proposition}
\newtheorem{defn}[thm]{Definition}
\newtheorem*{conjecture*}{Conjecture}
\newtheorem{problem}[thm]{Problem}
\newtheorem{notation}[thm]{Notation}
\newtheorem{example}[thm]{Example}
\newtheorem{question}[thm]{Question}
\theoremstyle{remark}
\newtheorem{rem}[thm]{Remark}
\theoremstyle{plain}
\newtheorem{exa}[thm]{Example}

\makeatother

\begin{document}

\title[On skew braces]{On skew braces (with an appendix by N. Byott and L. Vendramin)}

\begin{abstract}
	Braces are generalizations of radical rings, introduced by Rump to study
	involutive non-degenerate set-theoretical solutions of the Yang--Baxter equation (YBE).
	Skew braces were also recently introduced as a tool to study not
	necessarily involutive solutions.  Roughly speaking, skew braces provide
	group-theoretical and ring-theoretical methods to understand solutions of
	the YBE.  It turns out that skew braces appear in many different contexts,
	such as near-rings, matched pairs of groups, triply factorized groups,
	bijective $1$-cocycles and Hopf--Galois extensions. These connections and
	some of their consequences are explored in this paper. We produce several
	new families of solutions related in many different ways with rings,
	near-rings and groups.  We also study the solutions of the YBE that skew
	braces naturally produce.  We prove, for example, that the order of the
	canonical solution associated with a finite skew brace is even: it is two
	times the exponent of the additive group modulo its center. 
\end{abstract}


\author{A. Smoktunowicz}
\author{L. Vendramin}

\address{School of Mathematics, The University of Edinburgh, James Clerk Maxwell Building, The Kings Buildings, Mayfield Road EH9 3JZ, Edinburgh}
\email{A.Smoktunowicz@ed.ac.uk}

\address{IMAS--CONICET and Departamento de Matem\'atica, FCEN, Universidad de Buenos Aires, Pabell\'on~1, Ciudad Universitaria, C1428EGA, Buenos Aires, Argentina}
\email{lvendramin@dm.uba.ar}

\address{Department of Mathematics, College of Engineering,
Mathematics and Physical Sciences, University of Exeter, Exeter 
EX4 4QF U.K.}  
\email{N.P.Byott@exeter.ac.uk}

\maketitle
\setcounter{tocdepth}{1}
\tableofcontents

\section*{Introduction}

In~\cite{MR1183474} Drinfeld posed the problem of studying set-theoretical
solutions of the Yang--Baxter equation. Such solutions are pairs $(X,r)$, where
$X$ is a set and 
\[
r\colon
X\times X\to X\times X,
\quad
r(x,y)=(\sigma_x(y),\tau_y(x))
\]
is a bijective map such that
\[
	(r\times\id)(\id\times r)(r\times\id)=(\id\times r)(r\times\id)(\id\times r).
\]

The first two papers addressing this combinatorial problem were those of
Etingof, Schedler, Soloviev~\cite{MR1722951} and Gateva-Ivanova and Van den
Bergh~\cite{MR1637256}. Both papers considered involutive and non-degenerate
solutions. A solution is said to be \emph{involutive} if $r^2=\id_{X\times X}$
and it is said to be \emph{non-degenerate} if all the maps
$\sigma_x,\tau_x\colon X\to X$ are bijective. 

In~\cite{MR1722951}, Etingof, Schedler and Soloviev introduced the
\emph{structure group} $G(X,r)$  of a solution $(X,r)$ as the group with
generators in $\{e_x:x\in X\}$ and relations
$e_xe_y=e_{\sigma_x(y)}e_{\tau_y(x)}$, $x,y\in X$. They proved that $G(X,r)$
acts on $X$ and there is a bijective $1$-cocycle $G(X,r)\to\Z^{(X)}$ where
$\Z^{(X)}$ is the free abelian group on $X$.  Bijective $1$-cocycles are a
powerful tool for studying involutive set-theoretical solutions of the
Yang--Baxter equation; see for example~\cite{MR1722951,MR1848966}. 

Involutive solutions have been intensively studied; see for
example~\cite{MR2652212,MR3374524,MR2095675,MR2927367,MR2885602,MR2368074,MR2383056}.
In~\cite{MR2278047}, Rump introduced braces, a new algebraic structure that
turns out to be equivalent to bijective $1$-cocycles;
see~\cite{MR2584610,GI15,MR3291816}.  According to the definition given by
Ced\'o, Jespers and Okni\'nski in~\cite{MR3177933}, a brace is a triple
$(A,\cdot,+)$, where $(A,\cdot)$ is a group, $(A,+)$ is an abelian group and 
\[
a(b+c)+a=ab+ac
\]
holds for all $a,b,c\in A$. In this paper these braces will be called
\emph{classical} braces. It was observed by Rump that radical rings form an
important family of examples of braces.  This observation suggests using
ring-theoretical methods to study involutive set-theoretical solutions. Rump
also observed that a classical brace $A$ produces an involutive non-degenerate
solution:
\[
	r_A\colon A\times A\to A\times A,
	\quad
	r_A(a,b)=\left( ab-a, (ab-a)^{-1}ab \right).
\]
Moreover, the structure group $G(X,r)$ admits a canonical brace structure. This
brace structure over $G(X,r)$ is extremely important for understanding the
structure of involutive set-theoretical solutions.

The study of non-involutive solutions of the Yang--Baxter equation is also an
interesting problem with several applications in algebra and topology. Lu, Yan
and Zhu~\cite{MR1769723} and Soloviev~\cite{MR1809284} extended the main
results of~\cite{MR1722951} to non-involutive solutions.  As in the involutive
setting, one defines the structure group $G(X,r)$ and proves that there is a
bijective $1$-cocycle with domain $G(X,r)$ (now with values in a group which is
in general not isomorphic to a free abelian group).  These results suggest a
generalization of classical braces known as \emph{skew braces};
see~\cite{MR3647970}. 

Skew braces produce non-degenerate set-theoretical solutions; see
Theorem~\ref{thm:YB}. Moreover, the results of~\cite{MR1769723,MR1809284} can
now be translated into the language of skew braces. In particular, one obtains
that $G(X,r)$ admits a canonical skew brace structure and its associated
solution $r_{G(X,r)}$ satisfies a universal property; see
Theorem~\ref{thm:G(X,r)}.  

It is remarkable that skew braces have connections with other algebraic
structures such as groups with exact factorizations, Zappa--Sz\'ep products,
triply factorized groups, rings and near-rings, regular subgroups, Hopf--Galois
extensions.  As skew braces produce non-degenerate solutions, these connections
yield several new families of examples of solutions of the Yang--Baxter
equation associated with rings, near-rings and groups.

\medskip
This paper is organized as follows. In Section~\ref{preliminaries} we review the
definition and some basic properties of skew braces and some elementary
examples are given. In Section~\ref{examples} connections to other topics in
algebra are explored. We prove in Theorem~\ref{thm:factorization} that
factorizable groups are skew braces. As a corollary we prove that
Zappa-Sz\'ep product of groups and semidirect products of groups are skew
braces. Theorem~\ref{thm:factorization} is also used to construct skew braces
from Jacobson radical rings. In Theorem~\ref{thm:brace2T} we prove that
skew braces provide examples of triply factorized groups. In
Theorem~\ref{thm:triple} we translate a result of Sysak for triply factorized
groups into the language of skew braces. Based on this theorem, one easily finds
a connection between near-rings and skew braces; see
Proposition~\ref{pro:nr2brace}.  Several general constructions
of skew braces are stated, for example semidirect products, Zappa--Sz\'ep
products and wreath products of skew braces.  
The first two sections contain several new examples of skew braces. We summarize these examples in the 
following table:

\begin{table*}[h]
\begin{tabular}{c|c|c}
Additive group & Multiplicative group & Reference\tabularnewline
\hline
$\Sym_{3}$ & $C_{6}$ & Example~\ref{exa:s3c6}\tabularnewline
dihedral group & quaternion group & Example~\ref{exa:d8q8}\tabularnewline
$\Alt_{4}$ & $C_{3}\rtimes C_{4}$ & Example~\ref{exa:simple}\tabularnewline
$\GL(n,C)$ & $U(n)\times T(n)$ & Example~\ref{exa:QR}\tabularnewline
$\Alt_{5}$ & $\Alt_{4}\times C_{5}$ & Example~\ref{exa:a5a4c5}\tabularnewline
$\PSL(2,7)$ & $\Sym_{4}\times C_{7}$ & Example~\ref{exa:PSL27S4C7}\tabularnewline
\end{tabular}
\end{table*}

In Section~\ref{YB} the canonical
non-degenerate solution associated to a skew brace (constructed in
Theorem~\ref{thm:YB}) is studied.  We prove in Corollary~\ref{cor:biquandles}
that the solutions associated with skew braces are biquandles; hence skew
braces could be used to construct combinatorial invariants of knots.  In
Theorem~\ref{thm:depth} it is proved that the solution associated to a finite
skew brace is always a permutation of even order; and the order of this
permutation is computed explicitly in terms of the exponent of
a certain quotient the additive group of the skew brace.  In
Section~\ref{ideals} ideals of skew braces simple skew braces and skew braces
of finite multipermutation level are introduced.  Finally, in
Section~\ref{others} it is proved that skew braces are related to other
algebraic structures such as cycle sets (Theorem~\ref{thm:ccs}) and matched
pairs of groups (Theorem~\ref{thm:matched}). 

\subsection*{Notations and conventions} 

If $X$ is a set, we write $|X|$ to denote the cardinality of $X$ and $\Sym_X$
to denote the group of bijective maps $X\to X$. For $n\in\N$ the symmetric
group in $n$ letters will be denoted by $\Sym_n$, the alternating group in $n$ letters
by $\Alt_n$ and the cyclic group of order
$n$ by $C_n$. Usually we simply write $ab$ to denote the product $a\cdot b$. 

\section{Preliminaries}
\label{preliminaries}

Skew braces were first defined in ~\cite{MR3647970}. In this section we recall
the basic notions and properties of skew braces.

\begin{defn}
A \emph{skew brace} is a triple $(A,\cdot,\circ)$, where $(A,\cdot)$ and
$(A,\circ)$ are groups and the compatibility condition 
\begin{align}
    \label{eq:compatibility}
    &a\circ (bc)=(a\circ b)a^{-1}(a\circ c)
\end{align}
holds for all $a,b,c\in A$, where $a^{-1}$ denotes the inverse of $a$ with respect
to the group $(A,\cdot)$. 
The group $(A,\cdot)$ will be the
\emph{additive group} of the brace and $(A,\circ)$ will be the \emph{multiplicative
group} of the brace. A skew brace is said to be \emph{classical} if its additive
group is abelian.
\end{defn}

\begin{defn}
  Let $A$ and $B$ be skew braces. A map $f\colon A\to B$ is said to be a
  \emph{brace homomorphism} if $f(aa')=f(a)f(a')$ and $f(a\circ a')=f(a)\circ
  f(a')$ for all $a,a'\in A$. 
\end{defn}

\begin{rem}
 Skew braces form a category.  
\end{rem}

\begin{rem}
  \label{rem:0=1}
  It follows from~\eqref{eq:compatibility} that in every brace $A$ the neutral elements
  of $\Add{A}$ and $\Mul{A}$ concide.
\end{rem}

\begin{exa}
  \label{exa:trivial}
  Let $A$ be a group. Then $a\circ b=ab$ gives a skew brace. Similarly, the
  operation $a\circ b=ba$ turns $A$ into a skew brace. 
\end{exa}

\begin{exa}
	\label{exa:sd}
	Let $A$ and $M$ be groups and let $\alpha\colon A\to\Aut(M)$ be a
	group homomorphism. Then $M\times A$ with 
	\[
	(x,a)(y,b)=(xy,ab),
	\quad
	(x,a)\circ (y,b)=(x\alpha_a(y),ab)
	\]
	is a skew brace. Similarly, $M\times A$ with
	\[
	(x,a)(y,b)=(x\alpha_a(y),ab),\quad
	(x,a)\circ (y,b)=(xy,ba)
	\]
	is a skew brace. 
\end{exa}

\begin{exa}
	\label{exa:times}
	Let $A$ and $B$ be skew braces. Then $A\times B$ with 
	\[
		(a,b)(a',b')=(aa',bb'),\quad
		(a,b)\circ (a',b')=(a\circ a',b\circ b'),
	\]
	is a skew brace. 
\end{exa}

\begin{lem}{\cite[Corollary 1.10]{MR3647970}}
	\label{lem:lambda}
	Let $A$ be a skew brace. The map 
	\[
	\lambda\colon\Mul{A}\to\Aut\Add{A},\quad
	\lambda_a(b)=a^{-1}(a\circ b),
	\]
	is a group homomorphism.
\end{lem}

\begin{rem}
	If $A$ is a skew brace and $a\in A$, the inverse of $a$ with respect to
	$\circ$ is the element $\overline{a}=\lambda^{-1}_a(a^{-1})$. 
\end{rem}

Lemma~\ref{lem:lambda} justifies the following definition:

\begin{defn}
	\label{def:structure}
	Let $A$ be a skew brace. The \emph{crossed group} of $A$ is defined as
	the group $\Gamma(A)=(A,\cdot)\rtimes(A,\circ)$ with multiplication 
    \[
    (a,x)(b,y)=(a\lambda_x(b),x\circ y).
    \]
\end{defn}

\begin{lem}{\cite[Lemma 2.4]{Bachiller3}}
	\label{lem:mu}
	Let $A$ be a skew brace and let 
	\[
	\mu\colon\Mul{A}\to\Sym_A,\quad
	\mu_b(a)=\overline{\lambda_a(b)}\circ a\circ b. 
	\]
	Then 
	$\mu_1=\id$ and $\mu_{a\circ b}=\mu_b\mu_a$ for all $a,b\in A$.
\end{lem}

The following lemma was proved by Bachiller for classical braces,
see~\cite[Proposition 2.3]{MR3465351}. The same proof also works for skew
braces. 

\begin{lem}{\cite[Lemma 1.1.17]{BachillerTESIS}}
\label{lem:Bachiller}
Let $A$ be a group and $\lambda\colon A\to\Aut(A)$ be a map such that
\begin{equation}
\label{eq:lambda}
\lambda_{a\lambda_a(b)}=\lambda_a\lambda_b,\quad a,b\in A.
\end{equation}
Then $A$ with $a\circ b=a\lambda_a(b)$ is a skew brace. 
\end{lem}

\begin{proof}
	The first claim is~\cite[Corollary 1.10]{MR3647970}. For the second claim see~\cite[Lemma 1.1.17]{BachillerTESIS}.
\end{proof}

\begin{example}
  \label{exa:s3c6}
  Let $A=\Sym_3$ and $\lambda\colon A\to\Sym_A$ be given by
  \begin{align*}
    &\lambda_{\id}=\lambda_{(123)}=\lambda_{(132)}=\id,\\
    &\lambda_{(12)}=\lambda_{(23)}=\lambda_{(13)}=\text{conjugation by $(23)$}.
  \end{align*}
  It is easy to check that $\lambda_{a\lambda_a(b)}=\lambda_a\lambda_b$
  for all $a,b\in A$. Hence $A$ is a skew brace by
  Lemma~\ref{lem:Bachiller}.  Since the transposition $(12)$ has order
  six in the group $\Mul{A}$, it follows that $\Add{A}\simeq\Sym_3$ and
  $\Mul{A}\simeq C_6$.
\end{example}

The following lemma provides another useful tool for constructing skew braces.

\begin{lem}
	\label{lem:dual}
	Let $\Mul{A}$ be a group and $\lambda\colon A\to\Sym_A$ be a group
	homomorphism.  Assume that 
	$\lambda_a(1)=1$ for all $a\in A$ and that
	\begin{equation}
		\label{eq:dual}
		\lambda_a(b\circ\lambda^{-1}_b(c))=\lambda_a(b)\circ\lambda^{-1}_{\lambda_a(b)}\lambda_a(c)
	\end{equation}
	for all $a,b,c\in A$. Then $A$  with 
	$ab=a\circ\lambda^{-1}_a(b)$ is a skew brace. 
\end{lem}

\begin{proof}
	Note that Equation~\eqref{eq:dual} is equivalent to
	\begin{equation}
		\label{eq:better}
		\lambda_a^{-1}(bc)=\lambda_a^{-1}(b)\lambda_a^{-1}(c).
	\end{equation}
	We prove that the operation is associative:
	\begin{align*}
		a(bc) &= a\circ\lambda^{-1}_a(bc)
		=a\circ(\lambda^{-1}_a(b)\lambda^{-1}_a(c))\\
		&=a\circ\lambda^{-1}_a(b)\circ \lambda^{-1}_{a\circ \lambda^{-1}_a(b)}(c)
		=(ab)\circ\lambda^{-1}_{ab}(c)=(ab)c.
	\end{align*}

	The neutral element $1$ of $A$ is a right identity: $a1=a\circ\lambda^{-1}_a(1)=a\circ 1=a$. 
	The element $a^{-1}=\lambda_a(\overline{a})$ is a right inverse of $A$ since 
	\[
	aa^{-1}=a\circ\lambda^{-1}_a(a^{-1})=a\circ\lambda^{-1}_a\lambda_a(\overline{a})=a\circ\overline{a}=1.
	\]
	Therefore $\Add{A}$ is a group by~\cite[\S1.1.2]{MR1357169}. 

	The brace compatibility condition follows from Equation~\eqref{eq:better}:
	\begin{align*}
		&(a\circ b)a^{-1}(a\circ c)=(a\circ b)\lambda_a(c)=a\lambda_a(b)\lambda_a(c)=a\lambda_a(bc)=a\circ (bc).
	\end{align*}
	The lemma is proved. 
\end{proof}


\begin{defn}
	A skew brace $A$ is said to be a \emph{two-sided} skew brace if 
	\[
	(ab)\circ c=(a\circ c)c^{-1}(b\circ c)
	\]
	holds for all $a,b,c\in A$. 
\end{defn}

\begin{exa}
  Let $A$ be a skew brace with abelian multiplicative group. Then $A$ is
  a two-sided skew brace.
\end{exa}

\begin{exa}
  Let $n\in\N$ be such that $n=p_1^{a_1}\cdots p_k^{a_k}$, where the $p_j$ are
  distinct primes, all $a_j\in\{0,1,2\}$ and $p_i^m\not\equiv 1\pmod{p_j}$ for
  all $i,j,m$ with $1\leq m\leq a_i$. Then every skew brace of size $n$ is a
  two-sided classical brace, since every group of order $n$ is abelian, see for
  example~\cite{MR1786236}.  
\end{exa}

\begin{exa}
	\label{exa:d8q8}
	Let 
	\[
	A=\langle r,s:r^4=s^2=1,srs=r^{-1}\rangle
	\]
	be the dihedral group of eight elements and let
	\[
	B=\{1,-1,i,-i,j,-j,k,-k\}
	\]
	be the quaternion group of eight elements.  Let
	$\pi:B\to A$ be the bijective map given by 
	\begin{align*}
		1\mapsto 1 &, & -1\mapsto r^2 &,  & -k\mapsto r^3s &,&  k\mapsto rs &,\\
		i\mapsto s &, & -i\mapsto r^2s &, &  j\mapsto r^3 &, & -j\mapsto r &.
	\end{align*}
	A straightforward calculation shows that $A$ with 
	\[
	  x\circ y=\pi(\pi^{-1}(x)\pi^{-1}(y))
	\]
	is a skew brace with additive group $A$ and multiplicative group
	isomorphic to $B$. This skew brace is two-sided.
\end{exa}

The following proposition provides other examples:

\begin{pro}
	Let $A$ be a skew brace such that $\lambda_a(a)=a$ for all $a\in A$.
	Then $A$ is a two-sided skew brace.
\end{pro}

\begin{proof}
	First we notice that $a^{-1}=\overline{a}$ since
	$\overline{a}=\lambda^{-1}_a(a^{-1})=\lambda_a^{-1}(a)^{-1}=a^{-1}$. In particular, 
	\begin{equation}
		\label{eq:xoy}
		x\circ y=\overline{\overline{y}\circ\overline{x}}=(y^{-1}\circ x^{-1})^{-1}
	\end{equation}
	for all $x,y\in A$. Using~\eqref{eq:compatibility} and~\eqref{eq:xoy} one obtains that 
	\begin{align*}
		(ab)\circ c &= (c^{-1}\circ (b^{-1}a^{-1}))^{-1}\\
		&=\left( (c^{-1}\circ b^{-1})c(c^{-1}\circ a^{-1})\right)^{-1}\\
		&=(c^{-1}\circ a^{-1})^{-1}c^{-1}(c^{-1}\circ b^{-1})^{-1}\\
		&=(a\circ c)c^{-1}(b\circ c).
	\end{align*}
	This completes the proof.
\end{proof}

Now we show a non-classical skew brace that is not two-sided: 

\begin{exa}
	\label{exa:simple}
	Let $G$ be the group generated by the 
	permutations 
	\[
	(1263)(48ba)(57c9),
	\quad
	(145)(278)(39a)(6bc).
	\]
	Then $G$ is a group of order twelve isomorphic to $C_3\rtimes C_4$. 
	Let $\pi\colon G\to\Alt_4$ be the bijective map given by
	\begin{align*}
		&\id\mapsto\id, && (16)(23)(4b)(5c)(79)(8a)\mapsto (14)(23),\\
		& (145)(278)(39a)(6bc)\mapsto (234), && (1b564c)(29837a)\mapsto (143),\\
		& (154)(287)(3a9)(6cb)\mapsto (243), && (1c465b)(2a7389)\mapsto (142),\\
		& (1362)(4ab8)(59c7)\mapsto (13)(24), && (1263)(48ba)(57c9)\mapsto (12)(34),\\
		& (1a68)(253c)(49b7)\mapsto (132), && (186a)(2c35)(47b9)\mapsto (124),\\
		& (1967)(243b)(5ac8)\mapsto (134), && (1769)(2b34)(58ca)\mapsto (123).
	\end{align*}
	A straightforward calculation shows that $\Alt_4$
	with the operation
	\[
		\sigma\circ\tau=\pi(\pi^{-1}(\sigma)\pi^{-1}(\tau))
	\]
	is a skew brace. 
	
	Let $a=(14)(23)$ and $b=c=(234)$. Then 
	\[
		(12)(34)=(ab)\circ c\ne (a\circ c)c^{-1}(a\circ b)=(123),
	\]
	hence it is not two-sided.
\end{exa}


\subsection{Skew braces with nilpotent additive group}

Skew braces with nilpotent additive group are similar to classical braces.  It
was observed in \cite{smoktunowicz2} Sylow subgroups of the additive group of
finite classical braces are also braces. 

 \begin{thm}
 \label{thm:Sylow_is_brace}
	Let $A$ be a finite skew brace whose additive group $\Add{A}$ is nilpotent and 
	decomposes as $A=A_1\cdots  A_k$, where $A_j$ is a Sylow
	subgroup of order $p_j^{\alpha_j}$, $p_j$ is a prime number and
	$\alpha_j\geq1$. Then each $A_{i}$ is a skew brace.
\end{thm}

\begin{proof}
	It is enough to prove that the subgroup $A_1$ of $\Add{A}$ is a subgroup of
	$\Mul{A}$. Remark~\ref{rem:0=1} implies that $A_1\ne\emptyset$.  Let
	$a\in A$ and $b\in A_1$. Since $p_1^{\alpha_1}b=0$ and $\lambda_a$ is a
	group automorphism of $\Add{A}$,
	$0=\lambda_a(p_1^{\alpha_1}b)=p_1^{\alpha_1}\lambda_a(b)$. Hence
	$\lambda_a(b)\in A_1$ and $\lambda^{-1}_a(b)=\lambda_{\overline{a}}(b)\in
	A_1$. Therefore $a\circ b=a\lambda_a(b)\in A_1$ and
	$\overline{a}=\lambda^{-1}_a(a^{-1})\in A_1$ for all $a,b\in A_1$.
\end{proof}

\begin{cor}
\label{cor:Sylow_are_braces}
    Let $A$ be a finite skew brace whose additive group $\Add{A}$ is nilpotent and 
	decomposes as $A=A_1 \cdots  A_k$, where $A_j$ is a Sylow
	subgroup of order $p_j^{\alpha_j}$, $p_j$ is a prime number and
	$\alpha_j\geq1$. 
	Then each $A_{i_1}\cdots A_{i_l}$ is a skew brace.
\end{cor}

\begin{proof}
	It follows from Theorem~\ref{thm:Sylow_is_brace} and induction on $k$. 
\end{proof}

\begin{cor}
	\label{cor:is_solvable}
	Let $A$ be a finite skew brace whose additive group $\Add{A}$ is nilpotent.
	Then $\Mul{A}$ is solvable. 
\end{cor}

\begin{proof}
	By Corollary~\ref{cor:Sylow_are_braces}, for each prime $p$ there exists
	subgroup of $\Add{A}$ of order coprime with $p$. Thus the claim follows
	from Hall Theorem; see for example~\cite[\S9.1.8]{MR1357169}.
\end{proof}

\begin{rem}
  Corollary~\ref{cor:is_solvable} was proved by Byott in the context of 
  Hopf--Galois extensions; see~\cite[Theorem 1]{MR3425626}.
\end{rem}

We recall some questions from~\cite{BachillerTESIS}, see
also~\cite[\S1]{MR3425626}. 

\begin{question}
  \label{q:1}
	Let $A$ be a finite skew brace with solvable additive group. Is the
	multiplicative group solvable?
\end{question}

\begin{question}
  \label{q:2}
	Let $A$ be a finite skew brace with nilpotent multiplicative group. Is
	the additive group solvable?
\end{question}

\begin{rem}
  Partial results to Questions~\ref{q:1} and~\ref{q:2} can be found in the
  context of Hopf--Galois extensions; see for
  example~\cite{MR2011974,MR3425626}.
\end{rem}

\subsection{Bijective $1$-cocycles}
\label{1cocycles}

In this subsection we review the equivalence between skew braces and bijective
$1$-cocycles. 

Let $G$ and $A$ be groups such that $G$ acts on $A$ by automorphisms.  Recall
that a \emph{bijective $1$-cocycle} is an invertible map $\pi\colon G\to A$
such that
\begin{align*}
	\pi(gh)=\pi(g)(g\cdot \pi(h))
\end{align*}
for all $g,h\in G$.

\begin{example}
  The maps of Examples~\ref{exa:d8q8} and~\ref{exa:simple} 
  are bijective $1$-cocycles.
\end{example}

Let $\pi\colon G\to A$ and $\eta\colon H\to B$ be bijective $1$-cocycles.  A
\emph{homorphism} between these bijective $1$-cocycles is a pair $(f,g)$ of
group homomorphisms  $f\colon G\to H$, $g\colon A\to B$ such that
\begin{align*}
&\eta f=g \pi,\\
&g(h\cdot a)=f(h)\cdot g(a),&&a\in A,\;h\in G.
\end{align*}

Bijective $1$-cocycles form a category.  

For a given group $A$ let $\Bc(A)$ be
the full subcategory of the category of bijective $1$-cocycles with objects
$\pi\colon G\to A$ and let $\BrAD(A)$ be the full subcategory of the category of
skew braces with additive group $A$.

\begin{thm}{\cite[Proposition 1.11]{MR3647970}}
  \label{thm:1cocycle} Let $A$ be a group. The categories $\BrAD(A)$
  and $\Bc(A)$ are equivalent.  
\end{thm}

\begin{rem}
  In the context of classical braces, Theorem~\ref{thm:1cocycle} was implicit
  in the work of Rump; see~\cite{MR2278047,MR3291816} or~\cite{GI15}.  
\end{rem}

\begin{rem} In~\cite{MR1653340} Etingof and Gelaki give a method of
  constructing finite-dimensional complex semisimple triangular Hopf algebras.
  They show how any non-abelian group which admits a bijective 1-cocycle gives
  rise to a semisimple minimal triangular Hopf algebra which is not a group
  algebra.  \end{rem}



\section{Examples and constructions} \label{examples}

\subsection{Factorizable groups} 

For an introduction to the theory of factorizable groups we refer
to~\cite{MR1211633}.  Recall that a group $A$ \emph{factorizes} through two
subgroups $B$ and $C$ if $A=BC=\{bc:b\in B,c\in C\}$. The factorization is said
to be \emph{exact} if $B\cap C=1$.

The following proposition produces factorizable groups from
classical and skew braces:

\begin{pro}
  \label{pro:bracefactorization}
  Let $A$ be a skew brace. Assume that there exist subbraces $B$ and $C$ such
  that $\Add{A}$ admits an exact factorization through $\Add{B}$ and $\Add{C}$.
  If $\lambda_b(c)\in C$ for all $b\in B$ and $c\in C$, then $\Mul{A}$ admits
  an exact factorization through $\Mul{B}$ and $\Mul{C}$. 
\end{pro}

\begin{proof}
  The claim follows from the equality $a=bc=b\circ\lambda^{-1}_b(c)$. 
\end{proof}

\begin{example}
  Let $A$ be a classical brace (or more generally, a skew brace with nilpotent
  additive group). Assume that the group $\Add{A}$ decomposes as $A_1\cdots A_k$, where
  the $A_j$ are the Sylow subgroups of $\Add{A}$. 
  Let $I\subseteq\{1,\dots,k\}$, $B=\prod_{i\in I}A_i$ and $C=\prod_{i\not\in I}A_i$. 
  Then $\Mul{A}$ admits an exact factorization through $B$ and $C$ by 
  Corollary~\ref{cor:Sylow_are_braces} and Proposition~\ref{pro:bracefactorization}.
\end{example}

\begin{thm} 
  \label{thm:factorization} 
  Let $A$ be a group that admits an exact
  factorization through two subgroups $B$ and $C$. Then $A$ with \[ a\circ
  a'=ba'c,\quad a=bc\in BC,\,a'\in A, \] is a skew brace with multiplicative
  group isomorphic to $B\times C$ and additive group isomorphic to $A$.
\end{thm}

\begin{proof} The map $\eta\colon B\times C\to A$, $\eta(b,c)=bc^{-1}$, is
  bijective.  Since $\eta$ is bijective and $a\circ
  a'=\eta(\eta^{-1}(a)\eta^{-1}(a'))$, it follows that $(A,\circ)$ is a group
  isomorphic to the direct product $B\times C$. To prove that $A$ is a skew
  brace it remains to show~\eqref{eq:compatibility}. Let $a=bc\in BC$ and
  $a',a''\in A$. Then \begin{align*} (a\circ a')a^{-1}(a\circ a'')
    &=(ba'c)a^{-1}(ba''c)\\ &=ba'c(c^{-1}b^{-1})ba''c\\ &=ba'a''c\\ &=a\circ
    (a'a'').  \end{align*} This completes the proof.  \end{proof}

\begin{exa}[QR decomposition] 
	\label{exa:QR}
	Let $n\in\N$.  The group $\GL(n,\C)$ admits an
  exact factorization as through the subgroups $U(n)$ and $T(n)$, where $U(n)$
  is the unitary group and $T(n)$ is the group of upper triangular matrices
  with positive diagonal entries.  Therefore there exists a skew brace $A$ with
  $\Add{A}\simeq\GL(n,\C)$ and $\Mul{A}\simeq U(n)\times T(n)$.  \end{exa}

\begin{exa} 
	\label{exa:a5a4c5}The alternating simple group $\Alt_5$ admits an exact factorization
  through the subgroups \[ A=\langle (123),(12)(34)\rangle\simeq\Alt_4,\quad
  B=\langle(12345)\rangle\simeq C_5.  \] By Theorem~\ref{thm:factorization},
  there exists a skew brace with additive group $\Alt_5$ and multiplicative
  group $\Alt_4\times C_5$. Compare 
  with~\cite[Corollary 1.1(i)]{MR3425626}.
\end{exa}

\begin{exa} 
	\label{exa:PSL27S4C7}
  The simple group $\PSL(2,7)$ admits an exact factorization through
  the subgroups $A\simeq\Sym_4$ and $B\simeq C_7$. By
  Theorem~\ref{thm:factorization}, there exists a skew brace with additive
  group $\PSL(2,7)$ and multiplicative group $\Sym_4\times C_7$.  Compare 
  with~\cite[Corollary 1.1(ii)]{MR3425626}.
\end{exa}


Recall from~\cite{MR1321145} that a pair $(A,B)$ of groups is said to be
\emph{matched} if there are two actions \[
B\xleftarrow{\leftharpoonup}B\times A\xrightarrow{\rightharpoonup}A \] such
that \begin{align} \label{eq:matched1}&b\rightharpoonup (aa')=(b\rightharpoonup
  a)\left( (b\leftharpoonup a)\rightharpoonup a' \right),\\
  \label{eq:matched2}&(bb')\leftharpoonup a=(b\leftharpoonup (b'\rightharpoonup
  a))(b'\leftharpoonup a) \end{align} for all $a,a'\in A$ and $b,b'\in B$.  If
the quadruple $(A,B,\rightharpoonup,\leftharpoonup)$ form a matched pair of groups,
then $A\times B$ is a group with multiplication \begin{align*}
  (a,b)(a',b')=\left(a(b\rightharpoonup a'),(b\leftharpoonup a')b'\right),
\end{align*} where $a,a'\in A$ and $b,b'\in B$. The inverse of $(a,b)$ is \[
(a,b)^{-1}=(b^{-1}\rightharpoonup a^{-1},(b\leftharpoonup
(b^{-1}\rightharpoonup a^{-1}))^{-1}).  \] This group will be denoted by
$A\bowtie B$ and it is known as the \emph{biproduct} or the \emph{Zappa--Sz\'ep
product} of $A$ and $B$.


\begin{cor} \label{cor:biproduct} Let $A$ and $B$ be a matched pair of groups.
  Then the biproduct $A\bowtie B$ is a skew brace with \[
  (a,b)(a',b')=(a(b\rightharpoonup a'),(b\leftharpoonup a')b'),\quad (a,b)\circ
(a',b')=(aa',b'b), \] where $a,a'\in A$ and $b,b'\in B$.  \end{cor}

\begin{proof} It follows from Theorem~\ref{thm:factorization} since the
  biproduct $A\bowtie B$ admits an exact factorization through the subgroups
  $A\bowtie1\simeq A$ and $1\bowtie B\simeq B$.  \end{proof}

Theorem~\ref{thm:factorization} is useful to construct skew braces associated
with rings.

\begin{pro} \label{pro:agata} Let $R$ be a ring (associative, noncommutative),
  let $S$ be a subring of $R$ and let $I$ be a left ideal in $R$ such that
  $S\cap I=0$ and $R=S+I$.  Assume that $S$ and $I$ are Jacobson radical rings
  (for example nilpotent rings).  Then $R$ with the operation \[ a\circ
  b=a+b+ab \] is a group and $R=S\circ I$ is an exact factorization.  \end{pro}

\begin{proof} It is easy to prove that $\circ $ is associative. Moreover, since
  $S$ and $I$ are Jacobson radical rings, it follows that $(S, \circ )$ and
  $(I, \circ )$ are groups. 
	
	We claim that each $r\in R$ can be written as $r=a\circ b$ for some
	$a\in S$ and $b\in I$. Since $R=I+S$, one writes $r=i+s$ for some $s\in
	S$ and $i\in I$. Now let $\overline{s}$ be the inverse of $s$ in the
	group $(S, \circ)$.  Then \[ r=s\circ (\overline{s}\circ r) \] with
	$s\in S$ and $\overline{s}\circ r=\overline{s}\circ
	(i+s)=i+\overline{s}i\in I$.  Since $(S,\circ)$ and $(I, \circ )$ are
	groups and $R=S\circ I$, it follows that $(R, \circ)$ is a group.  The
	factorization $R=S\circ I$ is exact since $I\cap S=0$.  \end{proof}

%

Particular cases of Proposition~\ref{pro:agata} can be easily obtained as
factors of free algebras or as factors of differential polynomial rings. 

\begin{example} \label{exa:agata} Let $F$ be a field and let $P=F\langle x_{1}, \dots
  , x_{n}\rangle$ be the noncommutative (associative) polynomial ring in $n$
  noncommuting variables, and let $A$ be the subalgebra of $P$ consisting of
  polynomials which have zero constant term. Let $V$ be the linear space over
  $F$ spanned by $x_{1}, \dots , x_{n}$  and let $V_{1}$ and $V_{2}$ be  linear
  subspaces of $A$ such that   $V=V_{1}\oplus V_{2}$.  Let $Q\subseteq A$ be an
  ideal in $A$ such that $A^{m}\subseteq Q$ for some $m$, and denote $J=QA$
  (note that $J$ is an ideal in $P$).  Let $R=A/J$. Then \[ S=\{ a+J:a\in
  PV_{1}\}\subseteq R, \quad I= \{a+J:a\in PV_{2}\}\subseteq R, \] satisfy the
  assumptions of Proposition~\ref{pro:agata} and hence $(R, \circ)$ admits an
  exact factorization $R=S\circ I$.  \end{example}

\begin{example} Let $N$ be a nilpotent ring and $M$ be a left $N$-module. Let
  $R$ be the ring of matrices \[ \begin{pmatrix} N & M\\ 0 & 0 \end{pmatrix}
    =\left\{\begin{pmatrix} n & m\\ 0 & 0 \end{pmatrix}:n\in N,\,m\in M \right\},
  \] and \[ S=\begin{pmatrix} N & 0\\ 0 & 0 \end{pmatrix}\subseteq R, \quad
  I=\begin{pmatrix} 0 & M\\ 0 & 0 \end{pmatrix}\subseteq R.  \] Then $R$, $I$ and
  $S$ satisfy the assumptions of Proposition~\ref{pro:agata} and the group $(R,
  \circ )$ admits an exact factorization as $R=S\circ I$.  \end{example}


\begin{rem} Exactly factorizable groups give rise to a special class of Hopf
  algebras, see for example~\cite{MR1382029,MR1783929,MR0229061,MR611561}.  \end{rem} 


\subsection{Triply factorized groups}

In~\cite{MR705793} Sysak observed an interesting connection between radical
rings and triply factorized groups. This idea shows that skew braces produce
triply factorized groups.

Recall that a \emph{triply factorized group} is tuple $(G,A,B,M)$, where $G$ is
a group with subgroups $A$, $B$ and $M$ and such that $G=AM=BM=AB$ and $A\cap
M=B\cap M=1$. 

\begin{thm} 
	\label{thm:brace2T} 
	Let $X$ be a skew brace. Let $G=\Gamma(X)$, 
	$A=\Add{X}\times1$,
	$M=1\times\Mul{X}$ and $B=\{(x,x):x\in X\}$. Then $(G,A,B,M)$ is a triply
	factorized group such that $A\cap B=1$.  
\end{thm}

  \begin{proof} 
	  Clearly $G=AM$ and $A\cap M=A\cap B=B\cap M=1$.  Let us prove that $B$ is
	  a subgroup of $G$. Clearly $B$ is nonempty. For $x,y\in X$, using that
	  $\overline{y}=\lambda^{-1}_y(y^{-1})$, one obtains  
	  \[
	  (x,x)(y,y)^{-1}=(x,x)(\overline{y},\overline{y})=(x\circ
	  \overline{y},x\circ\overline{y})\in B.  
	  \] 
	  To prove that $G=BM$ notice that $(x,y)= (x,x)(1,\overline{x}\circ y)\in
	  BM$.  Similarly $(x,y)=(xy^{-1},1)(y,y) \in AB$, proves that $G=AB$.  
  \end{proof}


%

\begin{example} \label{exa:triply} Let $R$ be a  nilpotent ring  (associative,
  noncommutative), let $S$ be a subring of $R$ and let $I_{1}$ and $I_{2}$ be
  left ideals of $R$ such that \[ S\cap I_{1}=S\cap I_{2}=I_{1}\cap I_{2}=0,
    \quad R=S+I_{1}=S+I_{2}=I_{1}+I_{2}.  \] Proposition~\ref{pro:agata} with
    $A=S$, $B=I_1$ and $M=I_2$ implies that $(R, \circ )$ is a triply group
    factorized group: \[ R=A\circ B=A\circ M= B\circ M,\quad A\cap B=0.  \]
  \end{example}

Let us show a particular case of Example~\ref{exa:triply}.

\begin{example} Recall the notation from Example~\ref{exa:agata}. Let $n=2m$
  for some $m\in\N$ and let \begin{align*} V_{1}=\sum_{i=1}^{m} Fx_{i}, &&
    V_{2}=\sum _{i=m+1}^{2m}Fx_{i}, && V_{3}=\sum_{i=1}^{m} F(x_{i}+x_{m+i}).
  \end{align*} Now let 	\begin{align*} A=\{ a+J:a\in PV_{1}\}, && B=\{ a+J:a\in
      PV_{2}\}, && M=\{ a+J:a\in PV_{3}\}.  \end{align*}
    Proposition~\ref{pro:agata} implies that $(R, \circ )$ is a triply
    factorized group: \[ R=A\circ B=A\circ M= B\circ M,\quad A\cap B=0.  \]
  \end{example}

\begin{rem}
	Let $A$ be a skew brace. The multiplicative group $\Mul{A}$ with actions
	$x\rightharpoonup y=\lambda_x(y)$ and $x\leftharpoonup y=\mu_y(x)$ form a
	matched pair of groups, see Lemma~\ref{lem:brace2matched}. The biproduct $\Mul{A}\bowtie\Mul{A}$ has 
	multiplication
	\[
	(x,y)(x',y')=(x\circ \lambda_y(x'),\mu_{x'}(y)\circ y')
	\]
	and it is a triply factorizable group with $A=\Mul{A}\times 1$, $M=1\times
	\Mul{A}$ and $\Delta=\{(x,\overline{x}):x\in A\}$. The multiplication on
	$\Delta$ is given by
	\begin{align*}
		(a,\overline{a})(b,\overline{b})&=(a\circ\lambda_{\overline{a}}(b),\mu_b(\overline{a})\circ\overline{b})=(ab,\overline{ab})
	\end{align*}
	since $ab=a\circ\lambda_{\overline{a}}(b)$ and
	$a\circ\lambda_{\overline{a}}(b)\circ\mu_{b}(\overline{a})\circ\overline{b}=1$.
	There is a left action of $\Mul{A}$ on $\Delta$ given by 
	\[
	a\cdot (b,\overline{b})=(1,a)(b,\overline{b})(1,a)^{-1}=(\lambda_a(b),\overline{\lambda_a(b)})
	\]
	and the map $\Delta\rtimes\Mul{A}\to \Mul{A}\bowtie\Mul{A}$ given by
	$((a,\overline{a}),b)\mapsto (a,\overline{a}\circ b)$ is a group
	isomorphism.
\end{rem}

\begin{lem} \label{lem:inverse_cocycle} Let $(G,A,B,M)$ be a triply factorized
  group with $M$ normal in $G$ and $A\cap B=1$.  For each $m\in M$ there exists a unique
  $\gamma(m)\in A$ such that $m\gamma(m)\in B$.  Moreover, the map
  $m\mapsto\gamma(m)$ is bijective.  \end{lem}

\begin{proof} Since $G=AB=BA$ and $A\cap B=1$, for each $m\in M$ there is a unique
	$\gamma(m)\in A$ such that $m\gamma(m)\in B$, i.e. if $m=ba$, then $\gamma(m)=a^{-1}$. Similarly, $A\subseteq MB=BM$ and
  $M\cap B=1$ imply that for each $a\in A$ there is a unique $\pi(a)\in M$ such
  that $\pi(a)a\in B$, i.e. if $a=b_1m_1$, then $\pi(a)=b_1m_1^{-1}b_1^{-1}$. Now it follows $\pi(\gamma(m))=m$ for all $m\in M$ and
  that $\gamma(\pi(a))=a$ for all $a\in A$.  \end{proof}

The following result is~\cite[Proposition 21]{MR2799412} in the language of
skew braces:

\begin{thm}[Sysak] \label{thm:triple} Let $(G,A,B,M)$ be a triply factorized
  group such that $M$ is normal in $G$ and $A\cap B=1$. Then $M$ 
  with 
	  \[ m\circ m'=\gamma^{-1}(\gamma(m)\gamma(m')),  \] 
	  where $\gamma$ is the map of Lemma~\ref{lem:inverse_cocycle}, 
  is a skew brace such that $\Gamma(M)\simeq
  G$.  \end{thm}

  \begin{proof} 
	  For $m,m'\in M$ write $a=\gamma(m)$ and $a'=\gamma(m')$.
	  By
	  Lemma~\ref{lem:inverse_cocycle}, 
	  $m\circ m'=\gamma^{-1}(\gamma(m)\gamma(m'))$ 
	  defines a group structure over $M$
	  isomorphic to that of $A$.  Since $m(am'a^{-1})(aa')=(ma)(m'a')\in B$, it
	  follows that \[ m\circ m'=m(am'a^{-1}).  \]

	  Now $M$ is a skew brace since \begin{align*} (m\circ m')m^{-1}(m\circ
		  m'') &= m(am'a^{-1})m^{-1}m(am''a^{-1})\\ &=mam'm''a^{-1}\\ &=m\circ
		  (m'm'').  \end{align*}

		  Since $G=MA=AM$, a routine calculation proves $\Delta\colon
		  \Gamma(M)\to G$, $(m,x)\mapsto m\gamma(x)$, is a bijective group
		  homomorphism: 
		  \begin{align*}
			  \Delta((m,x)(n,y))&=\Delta(m\lambda_x(n),x\circ y)\\
			  &=m\lambda_x(n)\gamma(x\circ y)\\ &=m\lambda_x(n)\gamma(x)\gamma(y)\\
			  &=mx^{-1}(x\circ n)\gamma(x)\gamma(y)\\
			  &=mx^{-1}(x\gamma(x)n\gamma(x)^{-1})\gamma(x)\gamma(y)\\
			  &=m\gamma(x)n\gamma(y)\\ &=\Delta(m,x)\Delta(n,y).  \end{align*} This
			  completes the proof.
		  \end{proof}


\subsection{Near-rings}

This section is based on the work of Sysak on near-rings; see for
example~\cite[\S10]{MR2799412}.  However, the connection with skew braces and
all the examples in this section are new.

We refer to \cite{MR854275} for the basic theory of near-rings.  Recall that
\emph{near-ring} is a set $N$ with two binary operations \[ (x,y)\mapsto x+y,
\quad (x,y)\mapsto x\cdot y, \] such that $(N,+)$ is a (not necessarily
abelian) group, $(N,\cdot)$ is a semigroup, and $x\cdot (y+z)=x\cdot y+x\cdot
z$ for all $x,y,z\in N$. We assume that our near-rings contain a multiplicative
identity, denoted by $1$.

\begin{exa} Let $G$ be a (not necessarilly abelian) additive group and $M(G)$
  be the set of maps $G\to G$. Then $M(G)$ is a near-ring under the following
  operations: 
  \[ (f+g)(x)=f(x)+g(x),\quad (f\cdot g)(x)=g(f(x)),\quad f,g\in
    M(G),\,x\in G.  
  \] 
\end{exa}

A subgroup $M$ of $(N,+)$ is said to be a \emph{construction subgroup} if $1+M$
is a subgroup of the multiplicative subgroup $N^\times$ of units of $N$.

\begin{lem} \label{lem:construction} Let $N$ be a near-ring and $M$ be a
  construction subgroup of $N$. Then $(1+M)\cdot M\subseteq M$. In particular,
  $1+M$ acts on $M$ by left multiplication.  \end{lem}

\begin{proof} 
	Let $a,a'\in 1+M$. Then 
	\[
		-a'+a=-a'+1-1+a=-(-1+a')+(-1+a)\in M
	\]
	since $-1+a'\in M$ and $-1+a\in M$. Let $m,m'\in M$ and write $m=-1+a$ and
	$m'=-1+a'$ for some $a,a'\in 1+M$. Then 
	\[
		(1+m)\cdot m'=a\cdot (-1+a')=-a+a\cdot a'\in M
	\]
	since $a\in 1+M$ and $a\cdot a'\in 1+M$. 
\end{proof}

\begin{pro} \label{pro:nr2brace} Let $N$ be a near-ring and $M$ be a
  construction subgroup.  Then $M$ is a skew brace with \[ mm'=m+m',\quad
  m\circ m'=m+(1+m)\cdot m'.  \] \end{pro}

\begin{proof} 
By Lemma~\ref{lem:construction}, the operations are well-defined.
  For each $m\in M$ let $\lambda_m$ be the map $n\mapsto (1+m)\cdot n$. It is routine
  to verify that $\lambda\colon M\to\Aut(M)$, $m\mapsto\lambda_m$, is a
  well-defined map such that $\lambda_{m+\lambda_m(n)}=\lambda_m\lambda_n$.  By
  applying Lemma~\ref{lem:Bachiller}, the proposition is proved.  \end{proof}

\begin{rem} Proposition~\ref{pro:nr2brace} shows a connection between
  near-rings and skew braces. This connection then answers~\cite[Question
  1]{MR3574204}.  \end{rem}

%

If $N$ is a near-ring and $M$ is a construction subgroup of $N$,
Proposition~\ref{pro:nr2brace} implies that $M$ is a skew brace.  The following
is the translation of a theorem of Hubert in the language of skew braces: 

\begin{thm}[Hubert] \label{thm:Hubert} Let $A$ be a skew brace with
  multiplicative group isomorphic to $G$. The near-ring $M(G)$ contains a
  construction subgroup $M$ such that $\Gamma(A)\simeq\Gamma(M)$.  
  \end{thm}

\begin{proof} By Theorem~\ref{thm:brace2T}, the group $G=\Gamma(A)$ provides a
  triply factorized group $G=MA=MB=AB$ with $A\cap B=1$. Now ~\cite[Theorem
  2.9]{MR2331608} applies.  \end{proof}

\subsection{Nilpotent rings}

We now construct examples of skew braces related to nilpotent rings and
algebras.  These examples are influenced  by near ring theory and construction
subgroups.  The following result is inspired by \cite[Example 1.6]{MR854275}. 

\begin{lem} \label{lem:aga} Let $F$ be a finite field and let  $A$ be a
  commutative $F$-algebra such that $A=F+N$ where $N$ is a nilpotent subalgebra
  of $A$.  Let $S$ be the set of all functions $A\to A$ which can be written as
  polynomials from $N[x]$ (where two functions are equal if they have the same
  values).  Then $S$ with the operation \[ f(x)\bullet g(x)=f(x)+g(x+f(x)) \]
  is a group.  \end{lem}

\begin{proof} Direct calculations show that the operation is associative and
  that $f(x)=0$ is the identity element of $S$. It suffices to prove that each
  element in $S$ has a left inverse, i.e, for each $g(x)\in S$ there exists
  $f(x)\in S$ such that $f(x)=-g(x+f(x))$.  The map $f(x)$ can be obtained
  recursively as \[ f(x)=-g(x-g(x-g(x+g(\cdots(x-g(x))\cdots )))), \] where the
  number of brackets is equal to $n$ and $N^n=0$. Indeed, for any $p\in N[x] $,
  $-g(x-g(x-g(x+g(\cdots (x-g(x+p))\cdots ))))=f(x)$ because the element $p$
  will be multiplied by at least $n$ elements from $N$ in the left hand-side of
  this equation. Hence it will have zero value (where the left hand side has
  $n$ brackets). By substituting $p=-g(x)$ we get  that \[
  f(x)=-g(x-g(x-g(x+g(\cdots (x-g(x))\cdots )))), \] where the number of
  brackets is $n+1.$ Therefore, $-g(x+f(x))= f(x),$ as required.  \end{proof}

\begin{rem} The same construction of Lemma~\ref{lem:aga} works when $A$ is a
  noncommutative associative algebra. In this case instead of the polynomial
  ring $A[x]$ one takes the noncommutative polynomial ring, where the variable
  $x$ does not commute with the elements of $A$.  \end{rem}

We are now ready to present some examples of skew braces inspired by the
near-ring of functions $M(G)$ over a group $G$.

\begin{pro} \label{pro:aga} Let $F$ be a finite field and let  $A$ be a
  commutative $F$-algebra such that $A=F+N$ where $N$ is a nilpotent subalgebra
  of $A$.  Let $S$ be the set of all functions $A\to A$ which can be written as
  polynomials from $N[x]$.  Then $S$ with the usual addition and  \[
  f(x)\bullet g(x)=f(x)+g(x+f(x)), \] is a classical brace.  \end{pro}

\begin{proof} 
	By Lemma~\ref{lem:aga} it remains to show the brace compatibility
  condition: \begin{align*} f(x)\bullet
    (g(x)+h(x))-f(x)&=g(x+f(x))+h(x+f(x))\\ &=f(x)\bullet
    g(x)-f(x)+f(x)\bullet h(x).  \end{align*} This completes the proof.
\end{proof}

\begin{rem}
	Notice that if we consider $S$ to be the set of polynomial functions from
	$N[x]$ with zero constant terms, then Proposition~\ref{pro:aga} has a very
	short proof: since $S$ is nilpotent in the near-ring $M(A,+)$, it is a
	construction subgroup hence a skew brace by Proposition~\ref{pro:nr2brace}.  As the polynomial function $x$
	is the identity map, and hence the identity in the near-ring $M(A, +)$, we
	get $(f\bullet g)(x)=f(x)+g(x+f(x))$.
\end{rem}

\begin{cor} The sets $T=\{f\in S:f(1)=0\}$ and $\{f\in T:f(0)=0\}$ are
  subbraces of $S$.  \end{cor}

\begin{proof} It follows from Proposition~\ref{pro:aga}.  \end{proof}

\begin{lem} \label{lem:aga2} Let $F$ be a finite field and let  $A$ be a
  commutative $F$-algebra such that  $A=F+N$ where $N$ is a nilpotent
  subalgebra of $A$.  Let $S$ be the set of all functions $A\rightarrow A$
  which can be written as polynomials from $N[x]$ (where two functions are
  equal if they have the same values). Then $S$ with the operation \[f(x)\odot
  g(x)=f(x)\circ g(x\circ f(x)), \] where $a\circ b=a+b+ab$, $a,b\in A$, is a
  group.  \end{lem}

\begin{proof} It is easy to prove that $\odot $ is associative and that
  $f(x)=0$ is the identity element of $S$. To prove that $S$ is a group it
  suffices to show that every element in $S$ has a left inverse, i.e. that for
  every $g(x)\in S$ there is $f(x)\in S$ such that $f(x)\circ g(x\circ
  f(x))=0$, so \[ f(x)=-g(x\circ f(x))-f(x)\cdot g(x\circ f(x)).  \]
	
	Let $n$ be such that $N^{n}=0$ and let $t(x)=\sum
	_{i=1}^{n}(-1)^{i}g(x)^{i}$. Then \[ g(x)\circ t(x)=t(x)\circ g(x)=0,
	\] and hence \[ g(x\circ f(x))\circ t(x\circ f(x))=t(x\circ f(x))\circ
	g(x\circ f(x))=0.  \] Therefore, the equation $f(x)\circ g(x\circ
	f(x))=0$ is equivalent to \[ f(x)=t(x\circ f(x)) \] Now $f(x)=t(x\circ
	t(x\circ t(\cdots (x\circ t(x))\cdots )))$, where the number of
	brackets is equal to $n$.  \end{proof}

\begin{rem} The same construction of Lemma~\ref{lem:aga2} works when $A$ is a
  noncommutative associative algebra. In this case instead of the polynomial
  ring $A[x]$ one takes the noncommutative polynomial ring, where the variable
  $x$ does not commute with the elements of $A$.  \end{rem}

\begin{pro} 
	\label{pro:aga2}
	Let $F$ be a finite field and let  $A$ be an $F$-algebra such that
  $A=F+N$ where $N$ is a nilpotent subalgebra of $A$.  Let $S$ be the set of
  all functions $A\rightarrow A$ which can be written as a noncommutative
  polynomials from $N[x]$.  Then $S$ with the operations \[f(x)\odot
  g(x)=f(x)\circ g(x\circ f(x)), \quad (f\circ g)(x)=f(x)\circ g(x), \] is a
  skew brace.  \end{pro}

\begin{proof} 
	By Lemma~\ref{lem:aga2} it suffices to prove the compatibility
  condition.  Let $f(x)^{-1}$ denote the inverse of $f(x)$ in the group
  $(S,\circ )$. 
	\begin{align*} (f \odot (g\circ h))(x) &= f(x)\circ (g\circ h)(x\circ
	  f(x))\\ &= f(x)\circ g(x\circ f(x))\circ h(x\circ f(x))\\ &= (f \odot
	  g)(x)\circ f(x)^{-1} \circ (f \odot h)(x).  \end{align*} This
	completes the proof.  
\end{proof}

\begin{rem}
	Proposition~\ref{pro:aga2} can be obtained from
	Proposition~\ref{pro:nr2brace} when $S$ is the set of functions which are
	polynomial functions from $N[x]$ with zero constant term.
\end{rem}
    


\subsection{Matched pair of skew braces} \label{subsection:matched_braces}

The construction of matched pair of braces was first considered by
Bachiller~\cite{Bachiller2} for classical braces. 

\begin{defn}			
	A pair of skew braces $(A,B)$ is said to be
	\emph{matched} if there are group homomorphisms $\alpha\colon
	(A,\circ)\to\Aut(B,\cdot)$ and $\beta\colon (B,\circ)\to\Aut(A,\cdot)$ such
	that \begin{align}
		\label{eq:matched_braces1}
		\lambda_a^A\beta_b&=\beta_{\alpha_a(b)}\lambda_{\beta^{-1}_{\alpha_a(b)}}^A(a),\\
		\label{eq:matched_braces2}
		\lambda_b^B\alpha_a&=\alpha_{\beta_b(a)}\lambda_{\alpha^{-1}_{\beta_b(a)}}^B(b),\quad 
		a\in A,b\in B, 
	\end{align} 
	where $\lambda^A$ is the map of $A$ and $\lambda^B$ is
	the map of $B$.  
\end{defn}

\begin{defn} Given a matched pair $(A,B,\alpha,\beta)$ of skew braces, define
  the \emph{biproduct} $A\bowtie B$ as the set of ordered pairs $(a,b)\in
  A\times B$ with the operations \begin{align}
    \label{eq:m_matched1}&(a,b)(a',b')=(aa',bb'),\\
    \label{eq:m_matched2}&(a,b)\circ(a',b') =(\beta_b(\beta^{-1}_b(a)\circ
    a'),\alpha_a(\alpha^{-1}_a(b)\circ b')).  \end{align} \end{defn}

\begin{pro} 
	\label{pro:matched_braces} 
	Given a matched pair
  $(A,B,\alpha,\beta)$ of skew braces, the biproduct $A\bowtie B$ is a skew
  brace.
\end{pro}

\begin{proof} 
	We claim that
	\begin{align}
		\label{eq:alpha} \alpha_a(\alpha^{-1}_a(b)\circ y)=b\lambda^B_b\alpha_{\beta^{-1}_b(a)}(y)=b\circ \alpha_{\beta^{-1}_b(a)}(y),\\
		\label{eq:beta} \beta_b(\beta^{-1}_b(a)\circ x)=a\lambda^A_a\beta_{\alpha^{-1}_a(b)}(x)=a\circ \beta_{\alpha^{-1}_a(b)}(x).
	\end{align}
	We only prove~\eqref{eq:alpha}. Since $\alpha$ is a group homomorphism,
	using~\eqref{eq:m_matched1} one obtains that 
	\begin{align*}
		\alpha_a(\alpha^{-1}_a(b)\circ y)
		&=\alpha_a\left(\alpha^{-1}_a(b)\lambda^B_{\alpha^{-1}_a(b)}(y)\right)\\
		&=b\alpha_a\lambda^B_{\alpha^{-1}_a(b)}(y)
		=b\lambda^B_b\alpha_{\beta^{-1}_b(a)}(y)
		=b\circ \alpha_{\beta^{-1}_b(a)}(y).
	\end{align*}
	Then 
	\begin{align*}
		\lambda_{(a,b)}(a',b')
		&=(\lambda_a^A\beta_{\alpha^{-1}_a(b)}(a'),\lambda_b^B\alpha_{\beta^{-1}_b(a)}(b')).
	\end{align*}
	A direct calculation shows that $\lambda_{(a,b)}\in\Aut(A\times B)$ for all
	$a\in A$ and $b\in B$. Thus 
	by Lemma~\ref{lem:Bachiller} it suffices to prove that 
	\[
		\lambda_{(a,b)}\lambda_{(x,y)}(a',b')=\lambda_{(a,b)\circ (x,y)}(a',b').
	\]
	This is equivalent to prove the following two equalities:
	\begin{align}
		\label{eq:m1}\beta_b\lambda^A_{\beta^{-1}_b(a)}\beta_y\lambda^A_{\beta^{-1}_y(x)}(a') &= \beta_{\alpha_a(\alpha^{-1}_a(b)\circ x)}\lambda^A_{\beta^{-1}_{\alpha_a(\alpha^{-1}_a(b)\circ y)}\beta_b(\beta^{-1}_b(a)\circ x)}(a'),\\
		\label{eq:m2}\alpha_a\lambda^B_{\alpha^{-1}_a(b)}\alpha_x\lambda^B_{\alpha^{-1}_x(y)}(b') &= \alpha_{\beta_b(\beta^{-1}_b(a)\circ y)}\lambda^B_{\alpha^{-1}_{\beta_b(\beta^{-1}_b(a)\circ y)}\alpha_a(\alpha^{-1}_a(b)\circ y)}(b').
	\end{align}
	Let us prove~\eqref{eq:m1}. 
	Let $a''=\beta^{-1}_b(a)$ and $b''=y$. We first observe that
	\begin{align*}
		\beta^{-1}_{\alpha_a(\alpha^{-1}_a(b)\circ y)}\beta_b(\beta^{-1}_b(a)\circ x) 
		&= \beta^{-1}_{b\circ \alpha_{a''}(b'')}\beta_b(a''\circ x)\\
		&= \beta^{-1}_{\alpha_{a''}(b'')}\beta^{-1}_b(\beta_b(a''\circ x))\\
		&= \beta^{-1}_{\alpha_{a''}(b'')}(a''\circ x)\\
		&= \beta^{-1}_{\alpha_{a''}(b'')}(a''\lambda^A_{a''}(x))\\
		&= \beta^{-1}_{\alpha_{a''}(b'')}(a'')\beta^{-1}_{\alpha_{a''}(b'')}\lambda^A_{a''}(x)\\
		&= \beta^{-1}_{\alpha_{a''}(b'')}(a'')\lambda_{\beta^{-1}_{\alpha_{a''}(b'')}(a'')}\beta^{-1}_{b''}(x)\\
		&= \beta^{-1}_{\alpha_{a''}(b'')}(a'')\circ \beta^{-1}_{b''}(x).
	\end{align*}
	This equality and~\eqref{eq:alpha} imply that 
	\begin{align*}
		\beta_b\lambda^A_{\beta^{-1}_b(a)}\beta_y\lambda^A_{\beta^{-1}_y(x)}(a')&=\beta_b\beta_{\alpha_{a''}(b'')}\lambda^A_{\beta^{-1}_{\alpha_{a''}(b'')}(a'')}\lambda^A_{\beta^{-1}_{b''}(x)}(a')\\
		&=\beta_{b\circ \alpha_{a''}(b'')}\lambda^A_{\beta^{-1}_{\alpha_{a''}(b'')}(a'')\circ \beta^{-1}_{b''}(x)}(a')\\
		&=\beta_{b\circ \alpha_{a''}(b'')}\lambda^A_{\beta^{-1}_{\alpha_{a''}(b'')}(a'')\circ \beta^{-1}_{b''}(x)}(a')\\
		&=\beta_{b\circ \alpha_{a''}(b'')}\lambda^A_{\beta^{-1}_{\alpha_a(\alpha^{-1}_a(b)\circ y)}\beta_b(\beta^{-1}_b(a)\circ x)}(a')\\
		&=\beta_{\alpha_a(\alpha^{-1}_a(b)\circ y)}\lambda^A_{\beta^{-1}_{\alpha_a(\alpha^{-1}_a(b)\circ y)}\beta_b(\beta^{-1}_b(a)\circ x)}(a').
	\end{align*}
	The proof of~\eqref{eq:m2} is similar. 
  \end{proof}

%
%

\begin{defn} Let $A$ and $X$ be skew braces. A \emph{left action} of $A$ on $X$
  is a group homomorphism $(A,\circ)\to\Aut_{\mathrm{B}}(X)$, where
  $\Aut_{\mathrm{B}}(X)$ denotes the group of brace automorphisms of $X$.
\end{defn}

An easy consequence of Proposition~\ref{pro:matched_braces} is the construction
of semidirect product of skew braces.  Semidirect products of classical braces
were considered by Rump~\cite{MR2442072}.  

\begin{cor} \label{cor:ltimes} Let $A$ and $B$ be skew braces. Assume that
  there is a left action $\alpha$ of $A$ on $B$. Then $A\times B$ with the
  operations \[ (a,b)(a',b')=(aa',bb'),\quad (a,b)\circ (a',b')=(a\circ
    a',b\circ \alpha_a(b')), \quad
	\] is a skew brace. This skew brace structure over $A\times B$ will be
	denoted by $A\ltimes B$.  \end{cor}


\begin{cor} \label{cor:rtimes} Let $A$ and $B$ be skew braces. Assume that
  there is a left action $\beta$ of $B$ on $A$. Then $A\times B$ with the
  operations \[ (a,b)(a',b')=(aa',bb'),\quad (a,b)\circ
    (a',b')=(a\circ\beta_b(a'),b\circ b'),
	\] is a skew brace.  \end{cor}


\begin{cor} \label{cor:double} Let $A$ be a skew brace such that
  \begin{equation} \label{eq:double}
    \lambda_a\lambda_b=\lambda_{\lambda_a(b)}\lambda_a,\quad a,b\in A.
  \end{equation} Then $D(A)=A\ltimes A$ is a skew brace. The skew brace $D(A)$
  will be called the \emph{double} of $A$.  \end{cor}

\begin{proof} The brace $A$ acts on $A$ if and only if~\eqref{eq:double} holds.
  Thus the claim follows from Corollary~\ref{cor:ltimes}.  \end{proof}

Wreath products of classical braces were considered in~\cite[Corollary
3.5]{MR2584610}. The construction also works for skew braces: 

\begin{cor} \label{cor:wreath} 
	Let $A$ be a skew brace. Let $n\in\N$ and $B$ be skew brace such that
	$\Mul{B}\subseteq\Sym_n$. Then the wreath product $A\wr B=A^{\times
	n}\rtimes B$ is a skew brace.  
\end{cor}

\begin{proof} According to Example~\ref{exa:times}, $A^{\times n}=A\times\cdots\times A$
  ($n$-times) is a skew brace. Let $\delta\colon B\to\Aut_{\mathrm{B}}(A^n)$,
  $b\mapsto\delta_b$, where \[
  \delta_b(a_1,\dots,a_n)=(a_{b(1)},\dots,a_{b(n)}).  \] Then $B$ acts on $A^n$
  and hence then claim follows from Corollary~\ref{cor:rtimes}.  \end{proof}

\section{Solutions of the Yang--Baxter equation} \label{YB}

Skew braces produce non-degenerate solution of the YBE. 

\begin{thm}{\cite[Theorem 3.1]{MR3647970}} \label{thm:YB} Let $A$ be a skew
  left brace. Then 
  \begin{align*} 
	  \label{eq:braiding_operator} &r_A\colon
    A\times A\to A\times A, \\
    &r_A(a,b)=(\lambda_a(b),\mu_b(a))=
	(\lambda_a(b),\lambda^{-1}_{\lambda_a(b)}( (a\circ b)^{-1}a(a\circ b)), 
	\end{align*} is a non-degenerate solution of the Yang--Baxter
  equation. Furthermore, $r_A$ is involutive if and only if $ab=ba$ for all
  $a,b\in A$.  \end{thm}

\begin{rem} \label{rem:trick} Let $A$ be a skew brace and $r_A$ its associated
  solution. If one writes $r(a,b)=(u,v)$, then $a\circ b=u\circ v$ since \[
    \lambda^{-1}_{\lambda_a(b)}( (a\circ b)^{-1}a(a\circ
    b))=\overline{\lambda_a(b)}\circ a\circ b \] for all $a,b\in A$.  \end{rem}

A \emph{biquandle} is a non-degenerate set-theoretical solution $(X,r)$ of the
YBE such that there exists a bijection $t\colon X\to X$ such that
$r(t(x),x)=(t(x),x)$ for all $x\in X$.  Biquandles have applications in
classical and virtual knot theory, see for example~\cite{MR2542696}
and~\cite{MR2896084}.  

\begin{cor} \label{cor:biquandles} Let $A$ be a skew brace and $r_A$ its
  associated solution of the YBE.  Then $(A,r_A)$ is a biquandle.  \end{cor}

  \begin{proof} 
	  Let $a\in A$.  By Theorem~\ref{thm:YB}, $b=\lambda_a^{-1}(a)$ is the
	  unique element of $A$ such that $r(a,b)=(a,b)$. Similarly, 
	  $\mu^{-1}_a(a)\in A$ is the unique element of $A$ such that
	  $r_A(\mu^{-1}_a(a),a)=(\mu^{-1}_a(a),a)$. It follows that the map $A\to
	  A$, $a\mapsto\mu^{-1}_a(a)$, is bijective with inverse
	  $a\mapsto\lambda^{-1}_a(a)$. 
  \end{proof}

Let $(X,r)$ be a non-degenerate solution.  Recall that the structure group of
$(X,r)$ is defined as the group $G(X,r)$ with generators in $\{e_x:x\in X\}$
and relations $e_xe_y=e_ue_v$ whenever $r(x,y)=(u,v)$.  Let $\iota\colon X\to
G(X,r)$ be the canonical map, i.e., $\iota(x)=e_x$. In general, $\iota$ is not
injective: 

\begin{exa} Let $X=\{1,2,3,4\}$, $\sigma=(12)$ and $\tau=(34)$. Then $(X,r)$,
  $r(x,y)=(\sigma(y),\tau(x))$, is a non-degenerate solution of the YBE.  The
  canonical map $\iota\colon X\to G(X,r)$, $i\mapsto e_i$, is not injective
  since for example \[ e_1e_2=e_1e_1 \] and hence $e_1=e_2$.  \end{exa}

The following result is~\cite[Theorem 9]{MR1769723} in the language of skew
braces, see also~\cite[Theorem 2.7]{MR1809284}: 

\begin{thm} \label{thm:G(X,r)} Let $(X,r)$ be a non-degenerate solution of the
  YBE.  Then there exists a unique skew left brace structure over $G(X,r)$ such
  that \[ r_{G(X,r)}(\iota\times\iota)=(\iota\times\iota)r.  \]
  Furthermore, if $B$ is a skew left brace and $f\colon X\to B$ is a map such
  that $(f\times f)r=r_B(f\times f)$, then there exists a unique skew brace
  homomorphism $\phi\colon G(X,r)\to B$ such that $f=\phi\iota$ and
  $(\phi\times\phi)r_{G(X,r)}=r_B(\phi\times\phi)$.  \end{thm}

\begin{proof} 
  By~\cite[Theorem 9]{MR1769723} 
  and the equivalence between skew braces and bijective $1$-cocycles of
  Theorem~\ref{thm:1cocycle}, it remains to prove that 
  \[
    \phi(gh)=\phi(g)\phi(h)
  \]
  for all $g,h\in G(X,r)$. 
  Write $\lambda_B=\mu$. Since $\phi(\lambda_g(h))=\mu_{\phi(g)}\phi(h)$, 
  \begin{align*}
    \phi(gh)&=\phi(g\circ\lambda^{-1}_g(h))
    =\phi(g)\circ \phi(\lambda_g^{-1}(h))
    =\phi(g)\circ\mu^{-1}_{\phi(g)}\phi(h)
    =\phi(g)\phi(h).
  \end{align*}
  From this the claim follows.
\end{proof}

\begin{exa} Let $G$ be a group that admits an exact factorization through the
  subgroups $A$ and $B$. By Theorem~\ref{thm:factorization}, $G$ is a skew
  brace with additive group $G$ and multiplicative group $A\times B$.
  Theorem~\ref{thm:YB} shows that the map $r\colon G\times G\to G\times G$
  given by \[ r(g,h)=(b^{-1}hb,a_1^{-1}ahb b_1^{-1}), \] where $g=ab$ and
  $b^{-1}hb=a_1b_1$ for $a,a_1\in A$ and $b,b_1\in B$, is a non-degenerate
  set-theoretical solution of the YBE.  This is essentially the solution
  constructed by Weinstein and Xu in~\cite[Theorem 9.2]{MR1178147}.  \end{exa}

\begin{exa} Let $A=\Sym_3$. Then $A$ is a skew brace with $a\circ b=ba$.
  Clearly $\lambda_a(b)=a^{-1}ba$, $a,b\in A$ and the associated solution is \[
  r_A\colon A\times A\to A\times A,\quad r_A(a,b)=(a^{-1}ba,a).  \] The order
  of $r_A$ is twelve and the restriction of $r_A$ to the conjugacy class of
  involutions of $A$ has order three.  \end{exa}

\begin{exa} The skew brace of Example~\ref{exa:s3c6} produces a solution of
  order twelve.  This solution is isomorphic to $(X,r)$, where
  $X=\{1,2,\dots,6\}$ and $r(x,y)=(\sigma_x(y),\tau_y(x))$ is given by
  \begin{align*} &\sigma_1=\id, && \sigma_2=\id, && \sigma_3=(263), \\ &
    \sigma_4=(236), && \sigma_5=(263), &&\sigma_6=(236),\\ &\tau_1=\id,
    &&\tau_2=(36)(45), && \tau_3=(36)(45),\\ &\tau_4=\id, && \tau_5=\id, &&
    \tau_6=(36)(45).  \end{align*} \end{exa}

\begin{exa} The skew brace of Example~\ref{exa:d8q8} produces a solution of
  order four.  This solution is isomorphic to $(X,r)$, where
  $X=\{1,2,\dots,8\}$ and $r(x,y)=(\sigma_x(y),\tau_y(x))$ is given by
  \begin{align*} &\sigma_1=\id, && \sigma_2=(25)(47), &&\sigma_3=(38)(47),
    &&\sigma_4=(25)(38), \\ &\sigma_5=(25)(47), &&\sigma_6=\id,
    &&\sigma_7=(25)(38), &&\sigma_8=(38)(47),\\ &\tau_1=\id, &&
    \tau_2=(25)(38), &&\tau_3=(25)(38), &&\tau_4=\id,\\ &\tau_5=(25)(38),
    &&\tau_6=\id, &&\tau_7=\id, &&\tau_8=(25)(38).  \end{align*} \end{exa}

\begin{defn} 
  Let $A$ be a skew brace with additive group $G$. The \emph{depth}
  of $A$ is defined as the exponent of the group $G/Z(G)$.  
\end{defn}

\begin{exa} 
  Classical braces have depth one.  
\end{exa}

To study the depth of a skew brace we need the following lemma. 

\begin{lem} \label{lem:depth} Let $A$ be a skew brace and let $n\in\N$.  Then
  \begin{align} \label{eq:r^2n}r^{2n}(a,b)&=((a\circ b)^{-n}a(a\circ
    b)^n,\overline{(a\circ b)^{-n}a(a\circ b)^n}\circ a\circ b),\\
    \label{eq:r^2n+1}r^{2n+1}(a,b)&=((a\circ b)^{-n}a^{-1}(a\circ
    b)^{n+1},\overline{(a\circ b)^{-n}a^{-1}(a\circ b)^{n+1}}\circ a\circ b),
    \end{align} for all $n\geq0$.  Moreover, the following statements hold:
  \begin{enumerate} \item $r^{2n}=\id$ if and only if $ab^n=b^na$ for all
      $a,b\in A$.  \item $r^{2n+1}=\id$ if and only if $\lambda_a(b)=(a\circ
	b)^na(a\circ b)^{-n}$ for all $a,b\in A$.  \end{enumerate} \end{lem}

\begin{proof} It suffices to prove~\eqref{eq:r^2n} and~\eqref{eq:r^2n+1}.  We
  proceed by induction on $n$. The case $n=0$ is trivial for~\eqref{eq:r^2n}
  and~\eqref{eq:r^2n+1}. Assume that the claim holds for some $n>0$. If $n$ is
  even, by applying the map $r$ to Equation~\eqref{eq:r^2n} and using
  Remark~\ref{rem:trick} we obtain that \begin{align*} r^{2n+1}(a,b) &= r\left(
    (a\circ b)^{-n}a(a\circ b)^n, \overline{(a\circ b)^{-n}a(a\circ b)^n}\circ
    a\circ b \right)\\ &=\left( (a\circ b)^{-n}a^{-1}(a\circ b)^n(a\circ
    b),\overline{(a\circ b)^{-n}a^{-1}(a\circ b)^n(a\circ b)}\circ a\circ
    b\right)\\ &=\left( (a\circ b)^{-n}a^{-1}(a\circ b)^{n+1},\overline{(a\circ
      b)^{-n}a^{-1}(a\circ b)^{n+1}}\circ a\circ b\right).  \end{align*} Thus
    Equation~\eqref{eq:r^2n+1} holds. If $n$ is odd, a similar argument shows
    that~\eqref{eq:r^2n} holds.  \end{proof}

%

\begin{thm} \label{thm:depth} Let $A$ be a finite skew brace with more than one
  element. Then the order of $r_A$ is $2d$, where $d$ is the depth of $A$.
\end{thm}

\begin{proof} Let $n$ be such that $r^{2n+1}=\id$. By applying
  Lemma~\ref{lem:depth} one obtains that $a^{-1}(a\circ b)^{n+1}=(a\circ b)^na$
  for all $a,b\in A$. In particular, if $b=1$, then $a=1$, a contradiction. 

	Therefore we may assume that the order of the permutation $r_A$ is
	$2n$, where $n=\min\{k:b^ka=ab^k\;\forall a,b\in A\}$. Now one computes
	\begin{align*} n =\min\{k:b^k\in Z(A)\;\forall b\in A\} =\min\{ k:
	  (bZ(A))^k = 1\;\forall b\in A\} =d, \end{align*} and the theorem is
	proved.  \end{proof}


\begin{exa} \label{exa:trivial_center} Let $A$ be a finite skew brace such that
  its additive group has trivial center. Then the order of $r_A$ is equal to
  $2e$, where $e$ is the exponent of the additive group of $A$.  \end{exa}

\begin{exa} 
  \label{exa:2p} 
  Let $p$ be an odd prime number and let $A$ be a non-classical skew brace of
  size $2p$. Then the additive group of $A$ is isomorphic to the dihedral group
  $\D_{2p}$ of size $2p$.  Since $Z(\D_{2p})=1$ and the exponent of $\D_{2p}$
  is $2p$, the order of $r_A$ is $4p$.
\end{exa}

%
%

\section{Ideals and retractable solutions} \label{ideals}

Ideals of skew braces were defined in~\cite{MR3647970}.

\begin{defn} Let $A$ be a skew brace. A normal subgroup $I$ of $(A,\circ)$ is
  said to be an \emph{ideal} of $A$ if $aI=Ia$ and $\lambda_a(I)\subseteq I$
  for all $a\in A$.  \end{defn}

\begin{exa} Let $f\colon A\to B$ be a skew brace homomorphism. Then $\ker f$ is
  an ideal of $A$ since $f(\lambda_a(x))=\lambda_{f(a)}(f(x))=1$ for all
  $x\in\ker f$ and $a\in A$.  \end{exa}

An important example of an ideal is the socle. As in the classical
case, the socle is useful for studying the structure of skew braces and
multpermutation solutions.

\begin{exa} Let $A$ be a skew brace. Then the \emph{socle} \[ \Soc(A)=\{a\in
  A:a\circ b=ab,\;b(b\circ a)=(b\circ a)b\text{ for all $b\in A$}\} \] is an
  ideal of $A$ contained in the center of $(A,\cdot)$; see~\cite[Lemma
  2.5]{MR3647970}.  \end{exa}

\begin{lem}
  Let $A$ be a skew brace. Then $\Soc(A)=\ker\lambda\cap Z(A,\cdot)$. 
\end{lem}

\begin{proof}
  By~\cite[Lemma 2.5]{MR3647970} we only need to prove
  $\ker\lambda\cap\Z(A,\cdot)\subseteq\Soc(A)$. Let $a\in\ker\lambda\cap\Soc(A)$. It suffices to show that 
  $b(b\circ a)=(b\circ a)b$ for all $b\in B$. Since $a$ is central, 
  $\overline{b}a=a\overline{b}$ for all $b\in A$. This implies that 
  $\overline{b}\circ\left( b(b\circ a) \right)=\overline{b}\circ\left( (b\circ
  a)b \right)$ for all $b$ and the claim follows.
\end{proof}

\begin{lem}{\cite[Lemma 2.3]{MR3647970}} \label{lem:ideal} Let $A$ be a skew
left brace and $I\subseteq A$ be an ideal. Then the following properties hold:
\begin{enumerate} \item $I$ is a normal subgroup of $(A,\cdot)$.  \item $a\circ
      I=aI$ for all $a\in A$.  \item $I$ and $A/I$ are skew braces.
    \end{enumerate} \end{lem}

\begin{lem}
  \label{lem:1st}
  Let $f\colon A\to B$ be a surjective homomorphism of skew braces. then
  $A/\ker f\simeq B$. 
\end{lem}

\begin{proof}
  A routine calculation shows that $A/\ker f\to B$, $a\ker(f)\mapsto f(a)$, is a well-defined
  isomorphism of skew braces.
\end{proof}

The following proposition is a simple application of the transfer theory.
We will use the following theorem of Schur, see for example~\cite[\S10.1.3]{MR1357169}. 
If $H$ is a central subgroup of finite index $n$ in a group $G$,
the map $x\mapsto x^n$ is a group homomorphism since it is the transfer of $G$ 
into $H$.

\begin{pro} Let $A$ be a skew brace. Assume that the socle has finite index
  $n$.  Then the map $A\to A$, $a\mapsto a^n$, is a group homomorphism.
\end{pro}

\begin{proof} Since $\Soc(A)\subseteq Z\Add{A}$ by~\cite[Lemma 2.5]{MR3647970},
  the claim follows since $\Soc(A)$ has finite index in $G$. 
\end{proof}

\begin{defn} A skew brace is said to be \emph{simple} if $A\ne1$ and $1$ and
  $A$ are the only ideals of $A$.  \end{defn}

\begin{exa} Skew braces with a prime number of elements are simple.  \end{exa}

\begin{exa} Skew braces with simple multiplicative group are simple.  \end{exa}

\begin{exa} Skew braces with simple additive group are simple.  \end{exa}

\begin{exa} \label{exa:simpleissimple} The skew brace of
  Example~\ref{exa:simple} is simple since a nontrivial proper normal subgroup
  of $C_3\rtimes C_4$ have size three and a nontrivial proper normal subgroup
  of $\Alt_4$ have size four; see Lemma~\ref{lem:ideal}.  \end{exa}

The following problem arises naturally.

\begin{problem} Classify finite simple skew braces.  \end{problem}

\begin{rem} The problem of classifying classical simple braces is intensively
  studied, see for example~\cite{Bachiller2}.  \end{rem}

\begin{defn} 
  Let $A$ be a skew brace. The \emph{socle series} of $A$ is defined as the
  sequence 
  \[ 
    A_1 = A,\quad A_{n+1}=A_n/\Soc(A_n),\quad n\geq1.  
  \]
\end{defn}

\begin{lem}
  Let $A$ ba a skew brace. Let $S_1(A)=\Soc(A)$ and 
  \[
    S_{n+1}(A)=\{a\in A:(a\circ b)^{-1}ab\in S_{n}(A),\, [b,b\circ a]\in S_n(A)\;\forall b\in A\}
  \]
  for $n\geq1$, where $[x,y]=x^{-1}y^{-1}xy$ denotes the commutator of $x$ and
  $y$.  Then $A_{n+1}=A/S_{n}(A)$ for all $n\in\N$.
\end{lem}

\begin{proof}
  Notice that $a\in S_{n+1}(A)$ if and only if
  $abS_n(A)=(a\circ b)S_n(A)$ and 
  $b(b\circ a)S_n(A)=(b\circ a)S_n(A)$ for all $b\in A$. Now the claim follows
  by induction on $n$.
\end{proof}

\begin{defn} Let $A$ be a skew brace. It is said that $A$ has \emph{finite 
  multipermutation level} if there exists $n\in\N$ such that $A_n$ has only one
  element.  \end{defn}

\begin{example} Let $A$ be the skew brace of Example~\ref{exa:d8q8}. The socle
  $\Soc(A)$ of $A$ has two elements and hence $A_2=A/\Soc(A)$ is the trivial
  classical brace over $C_2\times C_2$. It follows that $\Soc(A_2)=A_2$ and
  hence $A$ has finite multipermutation level.  \end{example}

\begin{example} Let $A$ be the simple skew brace of
  Example~\ref{exa:simpleissimple}. Then $\Soc(A)=1$ and hence $A$ does not
  have finite multipermutation level.  \end{example}

Recall the construction of semidirect product of skew braces of
Corollary~\ref{cor:ltimes}.

\begin{thm}
  Let $A$ and $B$ be skew braces of finite multipermutation level. Let
  $C=A\ltimes B$ be a semidirect product of $A$ and $B$. Then $C$ has finite
  multipermutation level.
\end{thm}

\begin{proof}
By induction one proves that 
\[
  1\times S_n(B)\subseteq S_n(C)
\]
for all $n\in\N$.  Since $B$ has finite multipermutation level, there exists
$k\in\N$ $1\times B\subseteq 1\times S_k(B)\subseteq S_k(C)$. Since
$\Soc(A)\times 1\subseteq S_{k+1}(C)$, one proves by induction that
$S_n(A)\times1\subseteq S_{n+k}(C)$ for all $n\in\N$. Now let $l\in\N$ be such
that $S_l(A)=A$. Then $A\times1=S_l(A)\times 1\subseteq S_{k+l}(C)$ and hence
$S_{k+l}(C)=C$.
\end{proof}

\begin{thm} \label{thm:fpl=>nilpotent} Let $A$ be a skew brace of finite
  multipermutation level. Then $\Add{A}$ is nilpotent.  \end{thm}

\begin{proof} We proceed by induction on the size of $A$. If the order of $A$
  is a prime number, then $\Add{A}$ is nilpotent.  Now assume that the result
  holds for all skew braces of size $<|A|$.  Since $A/\Soc(A)$ is nilpotent by
  induction and $\Soc(A)$ is a central subgroup of $\Add{A}$, it follows that
  $A$ is nilpotent.  \end{proof}

\begin{rem} The converse of Theorem~\ref{thm:fpl=>nilpotent} does not hold. One
  example is the simple classical brace of size $24$ constructed
  in~\cite[Remark 7.2]{Bachiller2}. Another example: In the list of skew braces
  computed in~\cite{MR3647970} one can find a non-classical brace of size $16$
  with trivial socle and nilpotent additive group.  \end{rem}

\section{Skew braces and other algebraic structures} \label{others}

In Subsection~\ref{1cocycles} we reviewed the equivalence between skew braces
and bijective $1$-cocycles. In this section we state several equivalences
involving skew braces. 

\subsection{Skew cycle sets}

Recall that a \emph{cycle set} is a pair $(X,\bullet)$, where $X$ is a set and
$(a,b)\mapsto a\bullet b$ is a binary operation on $X$ such that each map
$\varphi_a\colon X\to X$, $\varphi_a(b)=a\bullet b$, is bijective, and 
\[
    (a\bullet b)\bullet (a\bullet c)=(b\bullet a)\bullet (b\bullet c)
\]
holds for all $a,b,c\in A$.

A \emph{linear cycle set} is a triple $(A,+,\bullet)$, where $(A,+)$ is an
abelian group, $(A,\bullet)$ is a cycle set, and
\[
a\bullet (b+c)=(a\bullet b)+(a\bullet c),\quad
(a+b)\bullet c=(a\bullet b)\bullet (a\bullet c)
\]
hold for all $a,b,c\in A$. 

Linear cycle sets were introduced by Rump in~\cite{MR2132760}.  Classical
braces are equivalent to linear cycle sets; see for example~\cite[Proposition
2.3]{MR3291816}. 

\begin{defn}
  A \emph{skew cycle set} is a triple $(A,\cdot,\bullet)$, where
  $\Add{A}$ is a (not necessarily
  abelian) group and $(a,b)\mapsto a\bullet b$ is a binary operation on
  $A$ such that each map $\varphi_a\colon X\to X$, $\varphi_a(b)=a\bullet b$,
  is bijective, and 
  \begin{align}
    \label{eq:ccs2}a\bullet (bc)&=(a\bullet b)(a\bullet c),\\
    \label{eq:ccs3}(ab)\bullet c&=(a\bullet b)\bullet (a\bullet c)
  \end{align}
  hold for all $a,b,c\in A$.
\end{defn}

\begin{rem}
  Let $A$ be a skew cycle set. It follows from~\eqref{eq:ccs3} that 
  \[
    (a\bullet b)\bullet (a\bullet c)=(b\bullet (b^{-1}ab))\bullet (b\bullet c)
  \]
  holds for all $a,b,c\in A$. 
\end{rem}

\begin{defn}
	Let $A$ and $B$ be skew cycle sets. A \emph{homomorphism} between
	$A$ and $B$ is a group homomorphism $f\colon A\to B$ such that 
	\[
	f(a\bullet a')=f(a)\bullet f(a')
	\]
	for all $a,a'\in A$. 
\end{defn}

\begin{notation}
	Let $A$ be a skew cycle set. The inverse operation of
	$\bullet$ will be denoted by $*$, i.e. $a\bullet b=c$ if and only if
	$a*c=b$, $a,b,c\in A$.
\end{notation}

\begin{lem}
	\label{lem:ccs:trick}
	Let $A$ be a skew cycle set. Then 
	\begin{align}
		\label{eq:ccs:trick1}a*(bc)&=(a*b)(a*c), \\
		\label{eq:ccs:trick2}(ab)*c&=a*( (a\bullet b)*c)
	\end{align}
	for all $a,b,c\in A$. 
\end{lem}

\begin{proof}
 	Let $a,b,c\in A$. 
	Since 
	$a\bullet\left( (a*b)(a*c) \right)=(a\bullet(a*b))(a\bullet(a*c))=bc$, 
	Equation~\eqref{eq:ccs:trick1} follows. Now let 
	\[
	d=(ab)\bullet c=(a\bullet b)\bullet (a\bullet c).
	\]
	Then $(ab)*d=c=a*\left( (a\bullet b)*d \right)$ and the lemma is proved. 
\end{proof}


Skew cycle sets form a category. 

For a group $A$ let $\Cs(A)$ be the full subcategory of skew cycle
sets whose objects are skew cycle set structures over $A$.

\begin{lem}
	\label{lem:brace2ccs}
	Let $A$ be a skew brace. Then the group $\Add{A}$ with 
	\[
	a\bullet
	b=\lambda_a^{-1}(b)=\overline{a}\circ (ab)
	\]
	is a skew cycle set.  Moreover, if $f\colon A\to A_1$ is a
	homomorphism of skew braces, then $f$ is a homomorphism of skew
	cycle sets.
\end{lem}

\begin{proof}
	Each map $\varphi_a:b\mapsto a\bullet b$ is bijective. Let
	$a,b,c\in A$. To prove~\eqref{eq:ccs2} one uses that $\lambda\colon
	(A,\circ)\to\Aut(A,\cdot)$ is a group homomorphism:
	\begin{align*}
		a\bullet(bc) &=\lambda^{-1}_a(bc)
		=\lambda_{\overline{a}}(bc)
		=\lambda_{\overline{a}}(b)\lambda_{\overline{a}}(c)
		=(a\bullet b)(a\bullet c).
	\end{align*}
	To prove~\eqref{eq:ccs3} we compute
	\begin{align*}
		(a\bullet b)\bullet (a\bullet c) & 
		=\lambda^{-1}_{\lambda^{-1}_a(b)}(\lambda^{-1}_a(c))
		=\lambda^{-1}_{a\circ \lambda^{-1}_a(b)}(c)
		=\lambda^{-1}_{ab}(c)
		=(ab)\bullet c.
	\end{align*}
	To prove that $f$ is a skew cycle set homomorphism one
	computes
	\[
		f(a\bullet b)=f(\overline{a}\circ (ab))
		=\overline{f(a)}\circ (f(a)f(b))
		=f(a)\bullet f(b).
	\]
	This finishes the proof.
\end{proof}

\begin{lem}
	\label{lem:ccs2brace}
	Let $A$ be a skew cycle set. Then $A$ with 
	$\lambda_a(b)=a*b$, where $a*b=c$ if and only if $a\bullet c=b$, is a
	skew brace. Moreover, if $f\colon A\to A_1$ is a skew cycle
	set homomorphism, then $f$ is a skew brace homomorphism. 
\end{lem}

\begin{proof}
	Let $\lambda\colon A\to\Sym_A$ be given by $a\mapsto \lambda_a$. Let
	$a,b,c\in A$.  First we notice that
	$\lambda_a(bc)=\lambda_a(b)\lambda_a(c)$ since 
	\[
		\lambda_a(bc)=a*(bc)=(a*b)(a*c)=\lambda_a(b)\lambda_a(c)
	\]
	by Lemma~\ref{lem:ccs:trick}, Equation~\eqref{eq:ccs:trick1}.

	To prove that $\lambda_{a\lambda_a(b)}(c)=\lambda_a\lambda_b(c)$ holds we
	use Lemma~\ref{lem:ccs:trick}, Equation~\eqref{eq:ccs:trick2} to obtain
	that
	\begin{align*}
	\lambda_a\lambda_b(c)&=a*(b*c)
	=a*\left( (a\bullet(a*b))*c \right)\\
	&=(a(a*b))*c
	=(a\lambda_a(b))*c
	=\lambda_{a\lambda_a(b)}(c).
	\end{align*}

	Now since $\lambda_a(bc)=\lambda_a(b)\lambda_a(c)$ and
	$\lambda_{a\lambda_a(b)}(c)=\lambda_a\lambda_b(c)$ hold, $A$ is a skew
	brace by Lemma~\ref{lem:Bachiller}.

	Finally the map $f$ is a skew brace homomorphism since
	\begin{align*}
		f(a\circ b)&=f(a(a*b))
		=f(a)f(a*b)\\
		&=f(a)\left( f(a)*f(b) \right)
		=f(a)\lambda_{f(a)}(f(b))
		=f(a)\circ f(b)
	\end{align*}
	for all $a,b\in A$. 
\end{proof}

Lemmas~\ref{lem:brace2ccs} and~\ref{lem:ccs2brace} yield the following result:

\begin{thm}
	\label{thm:ccs}
    Let $A$ be a group. The categories $\BrAD(A)$ and $\Cs(A)$ are equivalent.
\end{thm}


\subsection{Matched pairs of groups}
\label{matched_pair_of_groups}

For a given group $\Mul{A}$ let $\Mp(A)$ be the category with objects the
matched pairs $(A,A)$ such that 
\begin{equation}
	\label{eq:matched}
a\circ b=(a\rightharpoonup b)\circ (a\leftharpoonup
b)
\end{equation}
for all $a,b\in A$ and morphisms all group homomorphisms $f\colon A\to A$ such
that 
\[
f(a\rightharpoonup b)=f(a)\rightharpoonup f(b),
\quad
f(a\leftharpoonup b)=f(a)\leftharpoonup f(b)
\]
for all $a,b\in A$.

\begin{lem}
	\label{lem:brace2matched}
	Let $A$ be a skew brace. Then $(\Mul{A},\Mul{A})$ is a matched pair of groups with 
	$a\rightharpoonup b=\lambda_a(b)$ and $a\leftharpoonup b=\mu_b(a)$, $a,b\in A$.
\end{lem}

\begin{proof}
	Lemma~\ref{lem:lambda} proves that $\lambda$ is a left action and
	Lemma~\ref{lem:mu} proves that $\mu$ is a right action.
	Thus we need to prove that 
	\begin{align}
		\label{eq:matched1brace}a\rightharpoonup (b\circ b')=(a\rightharpoonup b)\circ ( (a\leftharpoonup b)\rightharpoonup b'),\\
		\label{eq:matched2brace}(a\circ a')\leftharpoonup b=(a\leftharpoonup (a'\rightharpoonup b))\circ (a'\leftharpoonup b)
	\end{align}
	hold for all $a,a',b,b'\in A$. 
	
	For $a,a',b\in A$ one obtains that
	\begin{align*}
		(a\leftharpoonup (a'\rightharpoonup b))\circ (a'\leftharpoonup b) 
		&=\overline{\lambda_a\lambda_{a'}(b)}\circ a\circ \lambda_{a'}(b)\circ\overline{\lambda_{a'}(b)}\circ a'\circ b\\
		&=\overline{\lambda_{a\circ a'}(b)}\circ a\circ a'\circ b\\
		&=(a\circ a')\leftharpoonup b.
	\end{align*}
	
	For $a,b,b'\in A$ one obtains that
	\begin{align*}
		(a\rightharpoonup b)\circ ( (a\leftharpoonup b)\rightharpoonup b') 
		&=\lambda_a(b)\circ\left(\overline{\lambda_a(b)}\circ a\circ b\rightharpoonup b'\right)\\
		&=\lambda_a(b)\circ\lambda_{\overline{\lambda_a(b)}}\lambda_{a\circ b}(b')\\
		&=\lambda_a(b)\lambda_{a\circ b}(b')\\
		&=a\rightharpoonup (b\circ b').
	\end{align*}
	This completes the proof.
\end{proof}

\begin{lem}
	\label{lem:matched2brace}
	Let $\Mul{A}$ be a group and $(A,A,\rightharpoonup,\leftharpoonup)$
	be a matched pair of groups such that $a\circ b=(a\rightharpoonup b)\circ
	(a\leftharpoonup b)$ for all $a,b\in A$. Then $A$ with
	\[
		ab=a\circ(\overline{a}\rightharpoonup b)
	\]
	is a skew brace.
\end{lem}

\begin{proof}
	For $a,b\in A$ write $\lambda_a(b)=a\rightharpoonup b$. Then $\lambda\colon
	A\to\Sym_A$, $a\mapsto\lambda_a$, is a well-defined group homomorphism.
	Equation~\eqref{eq:matched} implies that $\lambda_a(1)=1$ for all $a$. 
	Since 
	\begin{align*}
		\lambda_a(b\circ \lambda^{-1}_b(c))
		&=\lambda_a(b)\circ(\lambda_{a\leftharpoonup b}\lambda^{-1}_b(c))\\
		&=\lambda_a(b)\circ\lambda_{\overline{\lambda_a(b)}\circ a\circ b}\lambda^{-1}_b(c)\\
		&=\lambda_a(b)\circ\lambda^{-1}_{\lambda_a(b)}\lambda_a(c),
	\end{align*}
	the claim follows from Lemma~\ref{lem:dual}
\end{proof}

For a given group $A$, let $\BrMU(A)$ be the full subcategory of the category of
skew braces with multiplicative group $A$.  
Combining
Lemma~\ref{lem:brace2matched} and Lemma~\ref{lem:matched2brace} one gets the
following result:

\begin{thm}
\label{thm:matched}
    Let $A$ be a group. 
	The categories $\BrMU(A)$ and $\Mp(A)$ are equivalent. 
\end{thm}


\begin{rem}
	Theorem~\ref{thm:matched} is implicit in the work of Lu, Yan and Zhu,
	see~\cite[Theorem 2]{MR1769723} and~\cite{MR2024436}.  The result for
	classical braces was proved by Gateva-Ivanova; see~\cite[Theorem
	3.7]{GI15}.  Our proof of Theorem~\ref{thm:matched} is essentially that of
	Gateva-Ivanova. 
\end{rem}

\appendix
\section{Hopf--Galois extensions}
\begin{center}
  {\small\scshape (By N. Byott and L. Vendramin)}
\end{center}

In this appendix we review the connection between skew braces and Hopf--Galois
extensions.  
This connection was first observed by Bachiller in~\cite[\S2]{MR3465351}. 

Let $K$ be a field and let $H$ be a cocommutative Hopf algebra over $K$. An
$H$-module algebra $A$ over $K$ is an $H$-Galois extension of $K$ if the map 
\[
\theta \colon A\otimes_K A\to\Hom_K(H,A),\quad
\theta(a\otimes b)(h)=a(h\cdot b),
\]
is bijective. 

Let $K\subseteq L$ be a finite extension of fields. A
Hopf--Galois structure on $L/K$ consists of a Hopf algebra $H$ over $K$ and an
action of $H$ on $L$ such that $L$ is an $H$-Galois extension of $K$.
In~\cite{MR878476}, Greither and Pareigis showed how to find all Hopf-Galois structures when $L/K$ is separable. For simplicity, we consider only the case where $L/K$ is also normal, so that $L/K$ is a Galois extension in the classical sense. We then have:

\begin{thm}[Greither--Pareigis]
  \label{thm:GP}
	Let $K\subseteq L$ be a finite Galois field extension with group $G$.
	Then Hopf--Galois extensions on $L/K$ correspond bijectively to regular
	subgroups $A$ of $\Sym_G$ normalized by $G$, where $G$ is considered as
	a subgroup of $\Sym_G$ by the regular left representation.
\end{thm}

Recall that a subgroup $A$ of $\Sym_G$ is {\em regular} if, given any $g$, 
$h \in G$, there is a unique $a \in A$ with $a \cdot g=h$. The isomorphism 
class of $A$ in Theorem~\ref{thm:GP} is known as the \emph{type} of the
Hopf-Galois structure. Note that $|A|=|G|$, but in general $A$ and $G$ need 
not be isomorphic.  

In the situation of Theorem \ref{thm:GP}, the fact that $A$ acts regularly on $G$ enables us to define a 
bijection between $A$ and $G$, via which we may translate the left regular action of $G$ on itself into an action of $G$ on $A$. Thus $G$ becomes a regular subgroup of $\Sym_A$. It was observed by Childs \cite{MR990979} that the condition in Theorem \ref{thm:GP}, namely that $A$ is normalized by $G$, holds if and only if $G$ is contained in the subgroup $\Hol(A)$ of $\Sym_A$, where $\Hol(A)= A \rtimes \Aut(A)$ is the {\em holomorph} of $A$. The group operation in $\Hol(A)$ is given by 
\[
	(a,f)(b,g)=(af(b),fg),
\]
and an element $(b,g)\in H$ acts on $a \in A$ by $(b,g)\cdot a=bg(a)$. (Thus the first factor $A$ in $\Hol(A)$ is identified with left multiplications by elements of $A$.)

Childs' observation was used in \cite{MR1402555} to give a formula to count Hopf-Galois structures:
\begin{pro} \label{pro:HGS-count}
The number $e(G,A)$ of Hopf--Galois
structures of type $A$ on a Galois extension $L/K$ with group $G$ is given by 
\[
	e(G,A)=\frac{|\Aut(G)|}{|\Aut(A)|}f(G,A),
\]
where $f(G,A)$ is the number of regular subgroups of $\Hol(A)$ that are
isomorphic to the group $G$. 
\end{pro}

We now sketch the proof of this, partly following the exposition in \cite[\S7]{MR1767499}, in order to elucidate the relationship between Hopf-Galois structures and skew braces.  

To begin with, we consider $G$ and $A$ as abstract groups, i.e.~given without any 
actions on each other. Let $\lambda_G : G \to \Sym_G$ be the left regular 
representation: $\lambda_G(g) \cdot h =gh$ for $g$, $h \in G$. We will 
call $\alpha : A \to \Sym_G$ a {\em regular embedding} if $\alpha$ is 
an injective group homomorphism whose image $\alpha(A) \subseteq \Sym_G$ 
is regular on $G$. A regular embedding $\alpha : A \to \Sym_G$ induces a bijection
$$ \alpha_* : A \to G, \qquad \alpha_*(a)=\alpha(a) \cdot 1_G. $$
Define $\beta : G \to \Sym_A$ by $\beta(g) = \alpha_*^{-1} \lambda_G(g)\alpha_*$. 
Then $\beta$ is also a regular embedding. In this way, we obtain a bijection 
from the set
$$ \mathcal{A} = \{\mbox{regular embeddings } \alpha : A \to \Sym_G\} $$
to the set
$$ \mathcal{G} = \{\mbox{regular embeddings } \beta : G \to \Sym_A\}, $$
whose inverse is obtained by the same construction with $A$ and $G$ interchanged. 
By the observation of Childs, this restricts to a bijection from 
$$ \mathcal{A}_0 = \{\alpha \in \mathcal{A} \, : \,  \alpha(A) \mbox{ is normalized by } G  \} $$
to
$$ \mathcal{G}_0 = \{ \beta \in \mathcal{G} : \beta(G) \subseteq \Hol(A)\}. $$

If $\alpha \in \mathcal{A}_0$ and $\phi \in \Aut(A)$, then also 
$\alpha \phi \in \mathcal{A}_0$. Thus $\Aut(A)$ acts on $A$ 
(from the right) by composition. This action is fixed-point-free: if 
$\alpha\phi=\alpha$ then $\phi=\id_A$. Moreover, for $\alpha$, 
$\alpha' \in \mathcal{A}_0$, we have 
$\alpha'(A)=\alpha(A) \Leftrightarrow \alpha '=\alpha \phi$ for 
some $\phi \in \Aut(A)$. Thus each regular subgroup 
$\alpha(A) \subseteq \Sym_G$ normalized by $G$ corresponds to an orbit 
of $\mathcal{A}_0$ under $\Aut(A)$, and each such orbit has 
cardinality $|\Aut(A)|$. By Theorem \ref{thm:GP}, the number of these 
subgroups is $e(G,A)$. Hence we have 
$|\mathcal{A}_0| = |\Aut(A)| e(G,A)$. A similar argument gives 
$|\mathcal{G}_0| = |\Aut(G)| f(G,A)$. As there is a bijective 
correspondence between $\mathcal{A}_0$ and $\mathcal{G}_0$, Proposition \ref{pro:HGS-count} follows. 

The action of $\Aut(A)$ on $\mathcal{A}_0$ by composition translates to an action 
on $\mathcal{G}_0$. Explicitly, if $\alpha \in \mathcal{A}_0$ 
corresponds to $\beta \in \mathcal{G}_0$, and $\phi \in \Aut(A)$, then 
$\alpha'=\alpha \phi$ corresponds to $\beta'$ where 
$\beta'(g)=\phi^{-1} \beta(g) \phi \in \Sym_A$. Thus the action of $\Aut(A)$ on 
$\mathcal{G}_0$ is by conjugation inside $\Sym_A$, and this 
action is again fixed-point-free. Two elements of $\mathcal{G}_0$ give 
rise to the same regular subgroup of $\Hol(A)$ if and only if they are in the same orbit
under this action. Thus the Hopf-Galois structures of type $A$ on $L/K$ correspond bijectively to the $\Aut(A)$-conjugacy classes of $\mathcal{G}_0$. 

One may check that the action of $\Aut(A)$ on $\mathcal{G}_0$ by conjugation commutes with the action 
of $\Aut(G)$ by composition. 

We now turn to the classification of skew braces. We have the following result from \cite[Proposition 4.3]{MR3647970}. 

\begin{pro}
	\label{pro:regular}
	Let $A$ be a group.  There exists a bijective correspondence between
	isomorphism classes of skew braces with additive group isomorphic to $A$ and classes
	of regular subgroups of $\Hol(A)$ under conjugation by elements of
	$\Aut(A)$. 
\end{pro}

\begin{proof}
  Let $\mathcal{B}(A)$ be the set of isomorphism classes of skew braces with
  additive group $A$ and let $\mathcal{R}(A)$ be the set of equivalence classes
  of regular subgroups of $\Hol(A)$ under conjugation by $\Aut(A)$. 
  
  Let $G$ be a
  regular subgroup of $\Hol(A)$. The regularity of $G$ implies that $\pi\colon
  G\to A$, $\pi(a,f)=a$, is bijective. Then $A$
  with the operation 
  \[
    a\circ b=\pi(\pi^{-1}(a)\pi^{-1}(b))=af(b)
  \]
  is a group isomorphic to $G$. Since
  \[
	a\circ (bc)=af(bc)=af(b)f(c)=af(b)a^{-1}af(c)=(a\circ b)a^{-1}(a\circ c),
  \]
  the set $A$ is a skew brace. A routine calculation shows that this
  correspondence induces a map $C\colon\mathcal{R}(A)\to\mathcal{B}(A)$.  

  Let $B\colon \mathcal{B}(A)\to\mathcal{R}(A)$ be given by
  $B(A)=\{(a,\lambda_a):a\in A\}$.  Routine calculations show that the map $B$ is
  well-defined and that  
  $B\circ C=\id_{\mathcal{R}(A)}$ and $C\circ
  B=\id_{\mathcal{B}(A)}$.
\end{proof}

\begin{rem}
  Proposition~\ref{pro:regular} was proved for classical braces by
  Bachiller~\cite[Proposition 2.3]{MR3465351}. 
\end{rem}

In terms of the preceding notation, the regular subgroups of $\Hol(A)$ which 
are isomorphic to $G$ correspond to orbits of $\mathcal{G}_0$ under $\Aut(A)$, 
and the isomorphism classes of skew braces with multiplicative group $G$ 
and additive group $A$ correspond to orbits of $\mathcal{G}_0$ under $\Aut(G) \times \Aut(A)$.  
We summarize the above discussion in the following result.

\begin{thm} 
Let $A$ and $G$ be finite groups of the same order, and let $\mathcal{G}_0$ 
be the set of regular embeddings $G \to \Hol(A)$. Then $\mathcal{G}_0$ 
admits commuting actions (from the right) of $\Aut(G)$ by composition 
and of $\Aut(A)$ by conjugation in $\Sym(A)$.

The set of Hopf-Galois structures of type $A$ on a Galois extension of fields 
with group $G$ corresponds bijectively to the set of orbits 
$\mathcal{G}_0 / \Aut(G)$, while the set of isomorphism classes of 
skew braces with multiplicative group $G$ and additive group $A$ corresponds 
bijectively to the set of orbits $\mathcal{G}_0 / (\Aut(G) \times \Aut(A))$. 

Hence there is a surjective map from this set of Hopf-Galois structures 
to this set of isomorphism classes of skew braces, induced by the canonical surjection
$$  \mathcal{G}_0 / \Aut(G) \twoheadrightarrow \mathcal{G}_0 / (\Aut(G) \times \Aut(A)).$$ 
\end{thm}

\begin{rem}
While each of the groups $\Aut(A)$ and $\Aut(G)$ acts without fixed points 
on $\mathcal{G}_0$, this is not true for $\Aut(G) \times \Aut(A)$, and 
the orbits under this group need not all have the same size. 
\end{rem}

In more concrete terms, in order to count either Hopf-Galois structures or skew braces, 
we need to determine the regular subgroups of $\Hol(A)$ isomorphic to $G$. To obtain the 
number of Hopf-Galois structures, we take the number of such subgroups and adjust 
by the factor $|\Aut(G)| / |\Aut(A)|$ as specified in Proposition \ref{pro:HGS-count}. To obtain the number of skew braces (up to isomorphism), we take the number of {\em orbits} of such subgroups under conjugacy by $\Aut(A)$. In general, these orbits are of different sizes, so there is no simple relationship between the number of Hopf-Galois structures and the number of skew braces. 





We illustrate the difference between counting Hopf-Galois structures and counting skew braces by means of an example.

\begin{example}
Let $G=C_{p^n}$ be the cyclic group of order $p^n$ for an odd prime $p$ and $n \in \N$. 
In this case, the Hopf-Galois structures were determined by Kohl \cite{MR1644203} (see
also~\cite[Theorem 9.1]{MR1767499}), and the classical braces were determined by 
Rump \cite{MR2298848}. If $A$ is a group of order $p^n$ (not necessarily abelian) 
such that $\Hol(A)$ contains an element of order $p^n$ then in fact $A$ is cyclic 
\cite[Theorem 4.4]{MR1644203}. Thus every Hopf-Galois structure on a cyclic field 
extension of degree $p^n$ is of cyclic type, and every skew brace with 
multiplicative group $C_{p^n}$ also has additive group $C_{p^n}$. In particular, 
there are no such skew braces beyond the classical braces found by Rump. Let $\sigma$ 
be a generator of $A=C_{p^n}$. Then $\Aut(A)=\{ \theta_u : u \in (\Z/p^n \Z)^\times \}$, 
where $\theta_u(\sigma)=\sigma^u$. Now any regular subgroup $G$ of $\Hol(A)$ must 
contain a unique element of the form $(\sigma,\theta_u)$, and it easy to check 
that this element generates $G$. Moreover, $u \equiv 1 \pmod{p}$ since $\theta_u$ 
must have $p$-power order. Hence there are $p^{n-1}$ possibilities for $u$. This 
gives $p^{n-1}$ distinct regular subgroups, and hence $p^{n-1}$ Hopf-Galois 
structures. To count the skew braces, we must consider the orbits of these subgroups 
under conjugacy by $\Aut(A)$. Now if $G=\langle (\sigma, \theta_u) \rangle$ then 
$\theta_v G \theta_v^{-1}$ is generated by $(\sigma^v,\theta_u)$. This subgroup 
is also generated by a unique element of the form $(\sigma, \theta_w)$. As $v$ varies, 
the possible values of $w$ are precisely those such that $u-1$ and $w-1$ are 
divisible by the same power of $p$. Hence we obtain $n$ skew braces (up to isomorphism), 
corresponding to $u=1$, $1+p^{n-1}$, $1+p^{n-2}$, $\ldots$, $1+p$. These skew braces 
have socles of size $p^n$, $p^{n-1}$, $\ldots$, $p$ respectively, and the 
corresponding orbits of regular subgroups under the action of $\Aut(A)$  
have sizes $1$, $p-1$, $p(p-1)$, $\ldots$, $p^{n-2}(p-1)$ respectively.
\end{example}

Having explained the connection between skew braces and Hopf-Galois structures, 
we restate a couple of known results for Hopf-Galois structures in terms of skew braces. 
The first example is the uniqueness result \cite[Theorem 1]{MR3647970}: 

\begin{thm}
  Let $n\in\N$.  There is a unique skew brace of size $n$ if and only if $n$
  and $\phi(n)$ are coprime, where $\phi$ denotes the Euler's totient function.
\end{thm}

The following result is~\cite[Theorem 2]{MR3425626}:

\begin{thm}
  Let $A$ be a finite skew brace with abelian multiplicative group. Then the
  additive group of $A$ is solvable.
\end{thm}


%
%

\begin{question}
  Let $A$ be a skew brace with multiplicative group isomorphic to $\Z$. Is the
  additive group of $A$ also isomorphic to $\Z$?
\end{question}

In~\cite[Algorithm 5.1]{MR3647970} a method to enumerate skew braces of small
size appears.  It is based on Proposition~\ref{pro:regular}. An easy
modification of the Magma~\cite{MR1484478} implementation of \cite[Algorithm
5.1]{MR3647970} allows us to enumerate Hopf--Galois extensions of small degree using 
Proposition~\ref{pro:HGS-count}. 


\begin{exa}
  In~\cite[Corollaries 6.3 and 6.4]{MR2030805} one finds that
  \[
	e(\Sym_3,\Sym_3)=e(C_6,\Sym_3)=2,\quad
	e(\Sym_3,C_6)=3,\quad
	e(C_6,C_6)=1.
  \]
  In~\cite[Corollary 6.6]{MR2030805} one finds that
  \begin{align*}
    e(C_7\rtimes C_3,C_7\rtimes C_3)&=16, & e(C_7\rtimes C_3,C_{21})&=7,\\
  e(C_{21},C_7\rtimes C_3)&=4, & e(C_{21},C_{21})&=1.
  \end{align*}
\end{exa}

Let $n\in\N$. Let $G_1,\dots,G_m$ be a complete set of representatives of
isomorphism classes of groups of order $n$. To record the number of
Hopf--Galois extensions of degree $n$, we constuct an $m\times m$
array $E(n)$ in which the $(i,j)$-entry is the number $e(G_i,G_j)$. 

\begin{exa}
  The arrays $E(8)$ and $E(12)$ are shown in Tables~\ref{tab:size8}
  and~\ref{tab:size12}, respectively.
\end{exa}

\begin{table}[h] 
  \caption{The number of Hopf--Galois extensions of fields of degree eight.}
  \begin{tabular}{|c|ccccc|}
    \hline 
    \rule{0pt}{2.5ex}
    & $C_8$ & $C_4\times C_2$ & $C_4\rtimes C_2$ & $Q_8$ & $C_2^3$\tabularnewline
    \hline
    $C_8$ &2 &  0&   2&   2 & 0\tabularnewline
    $C_4\times C_2$ &4 & 10&   6&   2 & 4\tabularnewline 
    $C_4\rtimes C_2$ &2 & 14&   6&   2 & 6\tabularnewline
    $Q_8$ &6 &  6&   6&   2 & 2\tabularnewline
    $C_2^3$ &0 & 42&  42&  14 & 8\tabularnewline
    \hline
  \end{tabular}
  \label{tab:size8}
\end{table}

\begin{table}[h] 
  \caption{The number of Hopf--Galois extensions of fields of degree twelve.}
  \begin{tabular}{|c|ccccc|}
    \hline
    & $C_3\rtimes C_4$ & $C_{12}$ & $\Alt_4$ & $C_6\rtimes C_2$ & $C_6\times C_2$\tabularnewline
    \hline
    $C_3\rtimes C_4$ & 2 & 3 & 12 & 2 & 3\tabularnewline 
    $C_{12}$ & 2 & 1 &  0 & 2 & 1\tabularnewline
    $\Alt_4$ & 0 & 0 & 10 & 0 & 4\tabularnewline
    $C_6\rtimes C_2$ & 14 & 9 &  0 &14 & 3\tabularnewline
    $C_6\times C_2$ & 6 & 3 &  4 & 6 & 1\tabularnewline
    \hline
  \end{tabular}
  \label{tab:size12}
\end{table}

The number $h(n)$ of Hopf--Galois structures of degree $n$ is
\[
  h(n)=\sum_{i=1}^m\sum_{j=1}^me(G_i,G_j). 
\]
Some values of $h(n)$ are shown in Table~\ref{tab:h(n)}. 

\begin{table}[h] 
    \caption{The number $h(n)$ of Hopf--Galois extensions of fields of degree $n$.}
    \begin{tabular}{|c|cccccccc|}
        \hline
        $n$    & 6 & 8 & 10 & 12 & 14  & 16 & 18 & 20 \tabularnewline
        $h(n)$ & 8 & 190 & 10 & 102 & 12  & 25168 & 289 & 166  \tabularnewline
        \hline
        $n$    & 21 & 22 & 24 & 25 & 26 & 27 & 28 & 30  \tabularnewline
        $h(n)$ & 28 & 16 & 5618 & 30 & 18 & 4329 & 128 & 80 \tabularnewline
        \hline
	$n$ & 34 & 36 & 38 & 40 & 42 & 44 & 45 & 46\tabularnewline
	$h(n)$ & 22 & 5980 & 24 & 8556 & 374 & 184 & 12 & 28 \tabularnewline
	\hline
    \end{tabular}
    \label{tab:h(n)}
\end{table}

\begin{problem}
  Compute $h(32)$. 
\end{problem}

\section*{Acknowledgements}

This research was supported with ERC advanced grant 320974.  The second-named
author is partially supported by PICT-2014-1376, MATH-AmSud 17MATH-01, ICTP and
the Alexander Von Humboldt Foundation. The authors thank Nigel Byott, Mart\'in Gonz\'alez Yamone, Timothy
Kohl, Victoria Lebed, Michael West and the referees for comments and corrections. 

\bibliographystyle{abbrv}
\bibliography{refs}


\end{document}